\documentclass[11pt,letterpaper]{article}
\usepackage[vmargin=1in,hmargin=0.8in]{geometry}
\usepackage{amsmath,amsthm,amssymb}
\usepackage{mathtools}
\usepackage{tikz,tikz-cd}
\usepackage[numbers]{natbib}
\usepackage{hyperref}
\usepackage{adjustbox}
\usepackage{enumerate}
\usepackage{float}
\usepackage{subcaption}

\usetikzlibrary{calc}
\usetikzlibrary{arrows.meta}
\usetikzlibrary{decorations.markings}

\newcommand{\N}{{\mathbb{N}}}
\newcommand{\R}{{\mathbb{R}}}
\newcommand{\Z}{{\mathbb{Z}}}
\newcommand{\F}{{\mathbb{F}}}

\newcommand{\tq}{\tilde{q}}
\newcommand{\tx}{\tilde{x}}
\newcommand{\ty}{\tilde{y}}
\newcommand{\tz}{\tilde{z}}

\newcommand{\tB}{\widetilde{B}}

\newcommand{\tL}{\widetilde{L}}
\newcommand{\tM}{\widetilde{M}}
\newcommand{\tN}{\widetilde{N}}
\newcommand{\tS}{\widetilde{S}}
\newcommand{\tW}{\widetilde{W}}
\newcommand{\hA}{\widehat{A}}
\newcommand{\hF}{\widehat{F}}
\newcommand{\hK}{\widehat{K}}
\newcommand{\hL}{\widehat{L}}
\newcommand{\hN}{\widehat{N}}
\newcommand{\cD}{{\mathcal{D}}}
\newcommand{\cH}{{\mathcal{H}}}
\newcommand{\cU}{{\mathcal{U}}}
\newcommand{\cP}{{\mathcal{P}}}
\newcommand{\sM}{{\mathsf{M}}}
\newcommand{\Rp}{\mathbb{R}^+}
\newcommand{\Zp}{\mathbb{Z}^+}

\newcommand{\setof}[1]{\left\{ {#1}\right\}}
\newcommand{\setdef}[2]{\left\{ #1 \, \middle| \, #2 \right\}}
\newcommand{\Int}{\mathop{\mathrm{int}}\nolimits}
\newcommand{\cl}{\mathop{\mathrm{cl}}\nolimits}
\newcommand{\Inv}{\mathop{\mathrm{Inv}}\nolimits}
\newcommand{\dom}{\mathop{\mathrm{dom}}\nolimits}
\newcommand{\SCH}{\mathop{\Sigma \mathrm{CH}}\nolimits}

\newcommand{\stoptime}{\sigma}
\newcommand{\incl}{\tilde{\iota}}
\newcommand{\susp}{\Sigma_\cH}
\newcommand{\inv}{^{-1}}

\tikzset{
  midarrow/.style={
    postaction={decorate},
    decoration={markings, mark=at position #1 with {\arrow{>}}}
  }
}

\numberwithin{equation}{section}
\numberwithin{figure}{section}
\numberwithin{table}{section}

\hypersetup{colorlinks=true, linkcolor=blue, citecolor=blue, urlcolor=blue,
  pdftitle={Conley Index Theory for a Class of Hybrid Systems},
  pdfauthor={Bernardo Rivas and William Kalies}
}

\theoremstyle{plain}
\newtheorem{thm}{Theorem}[section]
\newtheorem{lem}[thm]{Lemma}
\newtheorem{prop}[thm]{Proposition}
\newtheorem{cor}[thm]{Corollary}

\theoremstyle{definition}
\newtheorem{defn}[thm]{Definition}
\newtheorem{ex}[thm]{Example}

\theoremstyle{remark}
\newtheorem{rem}[thm]{Remark}

\title{Conley Index Theory for a Class of Hybrid Systems}
\author{Bernardo Rivas\thanks{Department of Mathematics and Statistics, University of Toledo, Toledo, OH 43606, USA.\newline \texttt{\detokenize{{bernardo.dopradorivas,william.kalies}@utoledo.edu}}.}
\and William Kalies\footnotemark[1]}
\date{}

\begin{document}

\maketitle

\begin{abstract}
This paper defines a homological Conley index for hybrid dynamical systems that satisfy the trapping guard condition. This is achieved by factoring through the hybrid suspension semiflow, which views such hybrid dynamics as a continuous semiflow. We prove that the isolated invariant sets of such a system correspond to the isolated invariant sets of its suspension semiflow, and that the index can be computed from hybrid index pairs in the original state space. Combined with classical Conley theory, this correspondence shows that the index is well-defined and satisfies the Wa\.zewski property, and that reconstruction results hold for fixed points and periodic orbits. Under an additional hypothesis, we show that it is robust to perturbations of the underlying semiflow and relate it to classical continuation results. We compute the index for several examples, showcasing how classical reconstruction theorems apply in this setting to characterize hybrid dynamics.


\end{abstract}

\section{Introduction}
\label{sec:introduction}

Hybrid dynamical systems provide a unified framework that includes continuous-time dynamics, discrete-time dynamics, and their interaction. Systems in which the two interact appear in a wide range of applications, including impact mechanics, switched control systems, robotic locomotion, and integrate-and-fire neuron models \cite{GoebelSanfeliceTeel2012,diBernardo2008Piecewise,Schaft2000Hybrid}. In such systems, a trajectory may follow a flow until it reaches a guard set, where a reset map sends it to a new state. The resulting discontinuities can produce phenomena with no counterpart for flows or maps, such as beating or Zeno phenomena. 

A central challenge in analyzing such systems is characterizing their long-term behavior and the structure of their attractors. In the classical setting, Conley theory addresses this problem through features of the dynamics that persist under perturbation~\cite{Conley1978}. The basic object is an isolating neighborhood: a compact set whose maximal invariant set is contained in its interior. The Conley index further proposes a topological characterization, which allows one to reconstruct dynamical information from algebraic topology and has been applied to show the existence of nontrivial invariant sets \cite{Conley1978}, heteroclinic orbits \cite{Conley1975}, fixed points \cite{McCord1988}, periodic orbits \cite{McCord1995}, and chaotic dynamics \cite{Day2019,MischaikowMrozek1995,Handbook2002}.

The global structure of the dynamics is described by Morse representations, whose Morse sets contain the chain-recurrent set, while outside the chain-recurrent set the dynamics is gradient-like~\cite{Conley1978}. In the hybrid setting, extensions of the decomposition theorem have been established for hybrid systems satisfying the trapping guard condition~\cite{Kvalheim2021} and for well-posed hybrid inclusions~\cite{Goebel2023}. This manuscript develops a Conley index theory for isolated invariant sets of hybrid systems that satisfy the trapping guard condition. 

To this end, we use the hybrid suspension semiflow~\cite{Kvalheim2021}. Briefly, its construction attaches an interval to each guard point and identifies the opposite endpoint with the reset image, replacing each reset by a trajectory segment of duration one (Section~\ref{sec:suspension}). To assign the classical Conley index to a hybrid isolated invariant set via this construction, one must establish that suspension preserves isolation. We prove a correspondence between the isolated invariant sets of the hybrid system and those of its suspension (Theorem~\ref{thm:IsolatedSet}), which extends to Morse representations (Proposition~\ref{prop:MorseCorrespondence}), and define hybrid index pairs that allow the computation of the Conley index in the suspension semiflow (Theorem~\ref{thm:SuspensionIndexPair}). These results provide the basis for defining the homological Conley index of a hybrid isolated invariant set and computing it directly from hybrid index pairs.

Once these results are established, the Wa\.zewski property and reconstruction results follow from classical Conley theory~\cite{Conley1978,McCord1988,McCord1995}. For continuation, we also establish joint continuity of the suspension semiflows under the uniform trapping guard condition with the state space, guard, and reset map fixed (Proposition~\ref{prop:FlowDependentContinuity}), so that classical continuation applies when an isolating neighborhood persists along the family (Theorem~\ref{thm:Continuation}). We compute the index for the examples in Section~\ref{sec:examples}, while general computational methods, including combinatorial representations of the hybrid suspension semiflow, are explored in~\cite{KaliesRivasWehbe2026}.

\subsection*{Organization}

The remainder of this paper is organized as follows. Section~\ref{sec:preliminaries} presents foundational concepts from classical Conley theory and hybrid dynamical systems. Section~\ref{sec:suspension} introduces the Trapping Guard Condition, details the construction of the hybrid suspension semiflow, and establishes the correspondence between the dynamics of the original hybrid system and its suspension. In Section~\ref{sec:conleyindex}, we formally define the Suspension Conley Index and prove its fundamental properties. Section~\ref{sec:examples} demonstrates the framework through applications to illustrative examples. Finally, Section~\ref{sec:conclusion} summarizes our contributions and discusses future directions.

\section{Preliminaries}
\label{sec:preliminaries}

In this section, we introduce the basic concepts of classical dynamical systems and hybrid dynamical systems that are used in the manuscript, and develop a Conley-theoretical approach to the hybrid setting. From here on, we consider $X$ to be a compact metric space and let $\Rp = [0,\infty)$ and $\Zp = \Z \cap \Rp$ denote the nonnegative reals and integers, respectively.

\subsection{Classical Dynamical Systems}
\label{sec:dynsys}
\begin{defn}
  \label{defn:DynamicalSystem}
  A \emph{local semiflow} in $X$ is a continuous function $\varphi : \cU \subseteq \Rp \times X \to X$ where $\cU$ is an open neighborhood of $\setof{0}\times X$ and $\varphi$ satisfies the semigroup property:
  \begin{enumerate}[(i)]
    \item $\varphi(0,x) = x$ for every $x \in X$, and
    \item
      for all $x \in X$ and $t,s \geq 0$, $(s+t,x) \in \cU$ if and only if $(t,x) \in \cU$ and $(s,\varphi(t,x))\in\cU$, and in that case
      \begin{equation}
        \label{eq:semigroup}
        \varphi(s+t,x) = \varphi(s,\varphi(t,x)).
      \end{equation}
  \end{enumerate}
  We denote it by $(X,\varphi)$ and refer to it as a \emph{semiflow} when $\cU = \Rp \times X$.
\end{defn}
\begin{defn}
  \label{defn:OmegaLimit}
  Consider a semiflow $\varphi : \Rp \times X \to X$. For any $U \subseteq X$, we write $\varphi(t,U) = \setdef{\varphi(t,x)}{x \in U}$. The \emph{$\omega$-limit set} of $U$ is
  \begin{equation}
    \omega(U,\varphi) \coloneqq \bigcap_{t \geq 0} \cl\left(\bigcup_{s \geq t}\varphi(s,U)\right).
  \end{equation}
\end{defn}
\begin{defn}
  \label{defn:ForwardInvariant}
  For $S \subseteq X$, define
  \[
    \Inv^+(S,\varphi) = \setdef{x \in X}{\varphi([0,\infty),x) \subseteq S}
  \]
  to be the set of points whose forward trajectories remain in $S$. A set $S$ is \emph{forward invariant} if $S = \Inv^+(S,\varphi)$, or equivalently, if $\varphi(t,S) \subseteq S$ for all $t \geq 0$.
\end{defn}
\begin{defn}
  \label{defn:FullTrajectory}
  A \emph{backward trajectory} through $x \in X$ is a map $\gamma: (-\infty, 0] \to X$ such that $\gamma(0)=x$ and $\varphi(s, \gamma(t)) = \gamma(t+s)$ for all $s \geq 0$ and $t \leq 0$ with $s+t \leq 0$. A \emph{full trajectory} through $x$ is a map $\gamma: \R \to X$ with $\gamma(0) = x$ such that $\varphi(s, \gamma(t)) = \gamma(t+s)$ for all $s \geq 0$ and $t \in \R$. For a full trajectory $\gamma: \R \to X$, the \emph{$\alpha$-limit set} and \emph{$\omega$-limit set} of $\gamma$ are
  \[
    \alpha(\gamma) \coloneqq \bigcap_{t \leq 0} \cl(\gamma((-\infty, t])), \qquad \omega(\gamma) \coloneqq \bigcap_{t \geq 0} \cl(\gamma([t, \infty))).
  \]
\end{defn}
\begin{defn}
  \label{defn:BackwardInvariant}
  A set $S$ is \emph{backward invariant} if $S = \Inv^-(S,\varphi)$ where
  \[
    \Inv^-(S,\varphi) = \setdef{x \in X}{\gamma((-\infty,0]) \subseteq S \text{ for some backward trajectory } \gamma \text{ through } x}.
  \]
\end{defn}
\begin{defn}
  \label{defn:InvariantSet}
  A set is \emph{invariant} if it is both forward and backward invariant.
  The \emph{maximal invariant set} in $K \subseteq X$ is
  \[
    \Inv(K,\varphi) \coloneqq \Inv^+(K,\varphi) \cap \Inv^-(K,\varphi),
  \]
  the set of points admitting full trajectories contained in $K$. A compact set $K$ is an \emph{isolating neighborhood} if $S = \Inv(K,\varphi) \subseteq \Int(K)$. In that case, $S$ is said to be an \emph{isolated invariant set}.
\end{defn}
A compact invariant set $A \subseteq X$ is an \emph{attractor} for $\varphi$ if there exists an open neighborhood $U \subseteq X$ of $A$ such that $\omega(U,\varphi) = A$.
\begin{defn}
  \label{defn:IndexPair}
  Let $S$ be an isolated invariant set. An \emph{index pair} for $S$ is a pair of compact sets $(N,L)$ such that $L \subset N$, and
  \begin{enumerate}[(i)]
    \item $S = \Inv(\cl(N\setminus L),\varphi)$ and $S \subseteq \Int(N\setminus L)$,
    \item $L$ is positively invariant in $N$, that is, if $x \in L$ and $\varphi([0,T],x) \subseteq N$ for some $T \geq 0$, then $\varphi([0,T],x) \subseteq L$,
    \item $L$ is an \emph{exit set} of the flow in $N$, that is, if $x \in N$ and $\varphi([0,\infty),x) \not\subseteq N$, then there exists $T \geq 0$ such that $\varphi([0,T],x) \subseteq N$ and $\varphi(T,x) \in L$.
  \end{enumerate}
\end{defn}

The index pair $(N,L)$ captures the essential topological structure of the flow near the isolated invariant set $S$, and each isolated invariant set admits an index pair \cite{Conley1978,Rybakowski1987}. To quantify this structure, we employ tools from algebraic topology, specifically homology.

Briefly, homology is a mathematical tool designed to measure the ``shape'' of a topological space by assigning a sequence of algebraic objects to it. For simplicity, we take coefficients over a (finite) field $\F$, which means these algebraic objects are vector spaces, denoted $H_k(X;\F)$. These vector spaces capture features of the space in different dimensions. Intuitively, the dimension of $H_0$ relates to the number of connected components, while the dimension of $H_k$ (for $k \geq 1$) detects $k$-dimensional ``holes''.

In the context of an index pair $(N,L)$, we are interested in the structure of the neighborhood $N$ relative to its exit set $L$. This is captured by \emph{relative homology}, denoted $H_k(N,L;\F)$. This measures the topological features of the quotient space $N/L$ obtained by collapsing the exit set $L$ to a single point when $(N,L)$ is a good pair~\cite{Handbook2002}.

\begin{defn}
  \label{defn:ConleyIndex}
  Consider a semiflow $(X,\varphi)$ and let $S$ be an isolated invariant set with index pair $(N,L)$.
  The \emph{(homology) Conley index} of an isolated invariant set $S$, denoted by $CH_*(S)$, is the relative homology of the index pair:
  \[
    CH_k(S) \coloneqq H_k(N,L;\F).
  \]
\end{defn}

We recall briefly some of the essential properties of the Conley index \cite{Conley1978,Salamon1985}. The index is \emph{well-defined}, meaning that for an isolated invariant set $S$, the relative homology for any two index pairs $(N,L)$ and $(N',L')$ is the same, that is, $H_*(N,L) \cong H_*(N',L')$. Another feature is its \emph{robustness} under perturbation: if a compact set $K$ remains an isolating neighborhood for a continuous family of semiflows $\varphi_\lambda$, $\lambda \in [0,1]$, then the sets $S_\lambda = \Inv(K,\varphi_\lambda)$ are related by continuation and their Conley indices are the same. Finally, the index has an existence criterion via the \emph{Wa\.zewski principle}, which states that if $CH_*(S)$ is nontrivial, the invariant set $S$ must be nonempty.

We illustrate with two examples what kind of indices one might expect from basic invariant structures.
\begin{ex}[Theorem 3.13 in \cite{Handbook2002}]
  \label{ex:conley}
  Let $S$ be a hyperbolic fixed point with an unstable manifold of dimension $n$. Then
  \[
    CH_k(S) =
    \begin{cases}
      \F & \text{if } k = n \\
      0 & \text{ otherwise.}
    \end{cases}
  \]
  This is illustrated in Figure~\ref{fig:conley_fp}.
\end{ex}
\begin{ex}[Corollary 3.17 in \cite{Handbook2002}]
  \label{ex:conley_periodic}
  Let $S$ be a hyperbolic invariant set that is diffeomorphic to a circle. Assume that $S$ has an oriented unstable manifold of dimension $n+1$. Then
  \[
    CH_k(S) =
    \begin{cases}
      \F & \text{if } k = n,n+1 \\
      0 & \text{ otherwise.}
    \end{cases}
  \]
  This is illustrated in Figure~\ref{fig:conley_periodic}.
\end{ex}
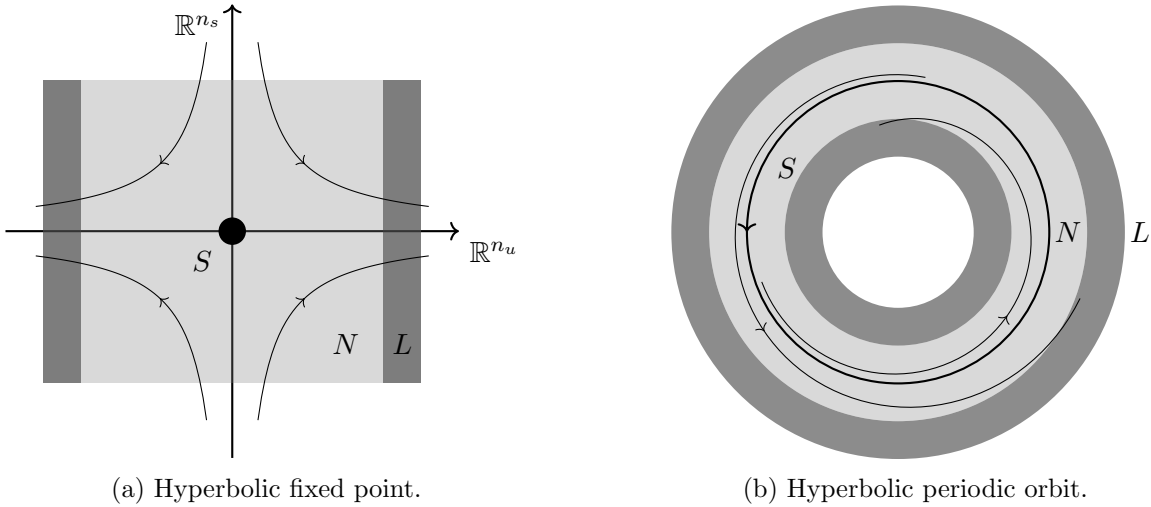
\begin{figure}[ht]
  \centering
  \begin{subfigure}[b]{0.48\textwidth}
    \centering
    \adjustbox{max width=\linewidth}{
      \begin{tikzpicture}
        \draw[->, thick] (-3,0) -- (3,0) node[below right] {$\R^{n_u}$};
        \draw[->, thick] (0,-3) -- (0,3) node[below left] {$\R^{n_s}$};
        \coordinate (A) at (-2.5, -2);
        \coordinate (B) at (2.5, 2);
        \fill[gray, opacity=0.3] (A) rectangle (B);
        \fill[darkgray, opacity=0.6] (A) rectangle (-2, 2);
        \fill[darkgray, opacity=0.6] (2, -2) rectangle (B);
        \foreach \sx in {1,-1}
        \foreach \sy in {1,-1}
        \draw[midarrow=0.5] plot[domain=0.34:2.6, samples=50] ({\sx*\x}, {\sy*0.85/\x});
        \node[circle, fill=black, label=below left:$S$] at (0,0) {};
        \node at (1.5, -1.5) {$N$};
        \node at (2.25, -1.5) {$L$};
    \end{tikzpicture}}
    \caption{Hyperbolic fixed point.}
    \label{fig:conley_fp}
  \end{subfigure}
  \hfill
  \begin{subfigure}[b]{0.48\textwidth}
    \centering
    \adjustbox{max width=\linewidth}{
      \begin{tikzpicture}
        \fill[darkgray, even odd rule,opacity=0.6] (0,0) circle (3.0) (0,0) circle (2.5);
        \fill[gray, even odd rule,opacity=0.3] (0,0) circle (2.5) (0,0) circle (1.5);
        \fill[darkgray, even odd rule,opacity=0.6] (0,0) circle (1.5) (0,0) circle (1.0);
        \draw[midarrow=0.5] plot[domain=0:260, samples=80] ({\x+80}:{2.03+0.05*exp(\x/110)});
        \draw[midarrow=0.5] plot[domain=0:260, samples=80] ({\x+200}:{1.97-0.05*exp(\x/110)});
        \draw[black,thick,midarrow=0.5] (0,0) circle (2.0);
        \node at (3.2,0.0) {$L$};
        \node at (2.25,0.0) {$N$};
        \node at (-1.47,0.85) {$S$};
    \end{tikzpicture}}
    \caption{Hyperbolic periodic orbit.}
    \label{fig:conley_periodic}
  \end{subfigure}
  \caption{Illustrations of index pairs $(N,L)$ for an isolated invariant set $S$. The set $L$ is shown in dark gray and $N\setminus L$ in light gray.}
  \label{fig:conley_example}
\end{figure}
\begin{rem}
  \label{rem:IndexIncomplete}
  Observe that computing the index does not yield complete knowledge about the maximal invariant set. The most straightforward example happens when $CH_*(S)=CH_*(\emptyset)$. This does not imply $S=\emptyset$ (e.g. $x'=x^2$, $S=\setof{0}$, $N=[-1,2]$, $L=[1,2]$). Similarly, one cannot conclude the existence of a periodic orbit merely based on the fact that $CH_*(S)$ coincides with the Conley index of a periodic orbit. However, under certain conditions, one can obtain information on the invariant set.
\end{rem}
\begin{thm}[\cite{McCord1988,Mrozek1988,McCord1995}]
  \label{thm:McCord}
  Let $(X,\varphi)$ be a semiflow and $S$ an isolated invariant set.
  \begin{enumerate}[(i)]
    \item If $CH_n(S) = \F$ for some $n \geq 0$ and $CH_k(S) = 0$ for all $k \neq n$, then $S$ contains a fixed point.
    \item If an isolating neighborhood of $S$ admits a \emph{Poincar\'e section} and if either
      \[
        \dim CH_{2n}(S) = \dim CH_{2n+1}(S) \ \text{for all } n \geq 0 \quad \text{or} \quad \dim CH_{2n}(S) = \dim CH_{2n-1}(S) \ \text{for all } n \geq 0,
      \]
      with $CH_*(S) \neq 0$, then $S$ contains a periodic orbit.
  \end{enumerate}
\end{thm}

\begin{defn}
  \label{defn:MorseRepresentation}
  Let $(X,\varphi)$ be a semiflow and let $S$ be a compact invariant set. A \emph{Morse representation} of $S$ is a finite collection
  \[
    \sM = \{M_1, \ldots, M_n\}
  \]
  of mutually disjoint, nonempty, compact invariant subsets of $S$, called \emph{Morse sets}, together with a partial order $\leq$ on $\sM$ such that:
  \begin{enumerate}
    \item for every $x \in S$, there exists $M_i \in \sM$ such that $\omega(x) \subseteq M_i$;
    \item for every full trajectory $\gamma: \R \to S$, there exist $M_i, M_j \in \sM$ with
      \[
        \alpha(\gamma) \subseteq M_j
        \qquad\text{and}\qquad
        \omega(\gamma) \subseteq M_i.
      \]
      If $\gamma(\R) \not\subset \bigcup_{k=1}^n M_k$, then $M_i < M_j$.
  \end{enumerate}
\end{defn}
For further background, we refer the reader to \cite{KaliesMischaikowVandervorst2014,KaliesMischaikowVandervorst2016,KaliesMischaikowVandervorst2022,Handbook2002}.

\subsection{Hybrid Dynamical Systems}
\label{subsec:hds}
Dynamical systems that involve both continuous and discrete systems simultaneously have several formulations, such as hybrid inclusions~\cite{GoebelSanfeliceTeel2012}, topological hybrid systems~\cite{Kvalheim2021}, and impulsive dynamical systems~\cite{Kaul1990,Bonotto2024}, to name a few. The theoretical framework of Section~\ref{sec:suspension} is defined for topological hybrid systems, so we introduce it, restricted to compact metric spaces, as follows.
\begin{defn}
  \label{defn:HybridSystem}
  Let $X$ be a compact metric space. A \emph{hybrid dynamical system} (HDS) in $X$ is a tuple $\cH = (X,\varphi,G,r)$ consisting of:
  \begin{enumerate}[(i)]
    \item a local semiflow $\varphi : \cU \subseteq \Rp \times (X \setminus G) \to X \setminus G$,
    \item a closed set $G \subseteq X$ called the \emph{guard set}, and
    \item a continuous function $r : G \to X$ called the \emph{reset map}.
  \end{enumerate}
\end{defn}
\begin{defn}
  \label{defn:ImpactTime}
  Let $\cH=(X,\varphi,G,r)$ be an HDS with local semiflow $\varphi : \cU \to X \setminus G$. The \emph{impact time} $\stoptime : X \to [0,\infty]$ is defined by
  \begin{equation}
    \label{eq:ImpactTime}
    \stoptime(x) = \inf \setdef{ t \geq 0 }{ (t,x) \notin \cU }.
  \end{equation}
  If the set is empty, we set $\stoptime(x) = \infty$.
\end{defn}
Note that $\stoptime(x) = 0$ for all $x \in G$.
\begin{ex}[Bouncing ball]
  \label{ex:BouncingBall}
  Consider a point mass moving vertically under the influence of gravity and bouncing off a fixed horizontal surface. The hybrid system $\cH = (X, \varphi, G, r)$ can be defined as follows:
  \begin{itemize}
    \item $X = \setdef{ (x,y) \in \mathbb{R}^2 }{ x \geq 0,\ \frac{1}{2}y^2 + gx \leq E_0 }$ for some $g>0$ and $E_0>0$,
    \item $\varphi$ is the local semiflow generated by $(\dot{x},\dot{y})=(y,-g)$,
    \item $G = \setdef{ (x,y) \in X }{ x = 0, y \leq 0}$,
    \item $r(0,y) = (0,-cy)$, where $c \in [0,1)$ is the coefficient of restitution.
  \end{itemize}
  Note that for each $(x,y) \in X$ and $t \in [0,\stoptime(x,y))$,
  \[
    \varphi(t,(x,y)) = \left( x + y t - \frac{1}{2}gt^2, y - gt \right) \text{ and } \stoptime(x,y) = \frac{y + \sqrt{y^2+2xg}}{g},
  \]
  where $\varphi(t,(x,y)) \to (0,-\sqrt{y^2+2xg})$ as $t \to \stoptime(x,y)$.
  \begin{figure}[ht]
    \centering
    \begin{tikzpicture}
      \draw[->, thick] (-0.5, 0) -- (4, 0) node[below right] {$x$ (position)};
      \draw[->, thick] (0, -2.5) -- (0, 2.5) node[below left] {$y$ (velocity)};
      \draw[line width=2pt, red] (0,0) -- (0,-2.5) node[left, black] {$G$};
      \draw[thick, midarrow=0.5] (0, 2.2) .. controls (3.5, 2.2) and (3.5, -2.2) .. (0, -2.2);
      \node[above right] at (2.5, 1.5) {$\varphi$};
      \draw[thick, dashed, midarrow=0.5] (0, -2.2) .. controls (-1.5, -1) and (-1.5, 1) .. (0, 1.5);
      \draw[thick,midarrow=0.5] (0, 1.5) .. controls (2, 1.5) and (2, -1.5) .. (0, -1.5);
      \draw[thick, dashed, midarrow=0.5] (0, -1.5) .. controls (-0.6, -1) and (-0.6, 1) .. (0, 0.75);
      \node[above left] at (-1.1, 0) {$r$};
    \end{tikzpicture}
    \caption{Hybrid dynamical system representing the bouncing ball. Solid lines represent the flow $\varphi$, and dashed lines represent the reset map $r$ upon hitting the guard set $G$ in red.}
    \label{fig:bouncing_ball_traj}
  \end{figure}
\end{ex}
\begin{defn}
  \label{defn:HybridTrajectory}
  Let $\cH=(X,\varphi,G,r)$ be an HDS and let $x \in X$. The \emph{hybrid trajectory} of $x$ is the collection of continuous functions $(\gamma_n \colon I_n \to X)_{n=0}^J$, with $J \in \Zp\cup\setof{\infty}$ maximal, that satisfy
  \begin{enumerate}[(i)]
    \item $t_0 = 0$, $\gamma_0(0) = x$, and $t_{n+1} = t_n + \stoptime(\gamma_n(t_n))$ for $0 \leq n \leq J$,
    \item $I_n = [t_n,t_{n+1}]$ with $t_{n+1} < \infty$ for $0 \leq n < J$, and $I_J = [t_J,t_{J+1})$ if $J < \infty$,
    \item $\gamma_n(t) = \varphi(t-t_n,\gamma_n(t_n))$ for all $0 \leq n \leq J$ and $t \in [t_n,t_{n+1})$,
    \item $\gamma_{n+1}(t_{n+1}) = r(\gamma_n(t_{n+1}))$ for $0 \leq n < J$.
  \end{enumerate}
  A trajectory is \emph{Zeno} if $J=\infty$ and $t_n \to T < \infty$ as $n \to \infty$.
\end{defn}

\begin{rem}  \label{rem:NonBlocking}
  The hybrid trajectory of each $x \in X$ exists and is unique, and $\gamma_n(t_{n+1}) \in G$ for $0 \leq n < J$ by Definition~\ref{defn:DynamicalSystem}(ii) and the openness of $\cU$. If $\cH$ satisfies the Trapping Guard Condition (Definition~\ref{defn:TrappingGuard}), then every hybrid trajectory is either defined for all $t \geq 0$ with finitely or infinitely many jumps or a Zeno trajectory where $t_n \to T < \infty$, that is, it is complete in the sense of~\cite{GoebelSanfeliceTeel2012}.
\end{rem}
In the spirit of connecting the notation to the classical theory, we shall denote the hybrid trajectory of $x$ by $\psi_x$ as in
\begin{align*}
  \psi_x \colon \dom(\psi_x) \coloneqq \bigcup_{n=0}^J (I_n \times\setof{n}) & \to X \\
  \psi_x(t,n) & = \gamma_n(t).
\end{align*}
The set $\dom(\psi_x)$ is a hybrid time domain~\cite{GoebelSanfeliceTeel2012} with $(s,m) \leq (t,n)$ if $s+m \leq t+n$. In a slight abuse of notation, we write the image of $\psi_x$ as
\[
  \psi([0,\infty),x) \coloneqq \setdef{\psi_x(t,n)}{(t,n) \in \dom(\psi_x)},
\]
and for $A \subseteq X$, we write $\displaystyle \psi([0,\infty),A)\coloneqq\bigcup_{x\in A}\psi([0,\infty),x)$.

\begin{defn}
  \label{defn:BackwardHybridTrajectory}
  Let $\cH=(X,\varphi,G,r)$ be an HDS and let $x \in X$. A \emph{backward hybrid trajectory} through $x$ is a collection of continuous functions $(\gamma_n \colon I_n \to X)_{n=-J}^0$, $J \in \Zp\cup\setof{\infty}$, that satisfy
  \begin{enumerate}[(i)]
    \item $t_0 = 0$, $\gamma_0(0) = x$, and $t_{n-1} \leq t_n$ for $-J < n \leq 0$,
    \item $I_n = [t_{n-1},t_n]$ for $-J < n \leq 0$, and $I_{-J} = (-\infty,t_{-J}]$ if $J < \infty$,
    \item $\gamma_n(t+s) = \varphi(s,\gamma_n(t))$ for all $-J \leq n \leq 0$, $t \in I_n$, and $s \geq 0$ with $t+s < t_n$,
    \item $\gamma_n(t_n) \in G$ and $\gamma_{n+1}(t_n) = r(\gamma_n(t_n))$ for $-J \leq n < 0$.
  \end{enumerate}
\end{defn}
\begin{rem}
  \label{rem:BackwardHybridTrajectory}
  Note that by Definition~\ref{defn:BackwardHybridTrajectory}(iii) and (iv), $\stoptime(\gamma_n(t)) = t_n - t$ for $-J \leq n < 0$ and $t \in I_n$. In contrast with the hybrid trajectory of $x$, which is unique (see Remark~\ref{rem:NonBlocking}), backward hybrid trajectories through $x$ may fail to exist or to be unique. We denote a backward hybrid trajectory through $x$ by $\beta_x$, where $\dom(\beta_x) \coloneqq \bigcup_{n=-J}^0 (I_n \times\setof{n})$ and $\beta_x(t,n) = \gamma_n(t)$, and its image by $\beta((-\infty,0],x)\coloneqq\beta_x(\dom(\beta_x))$, which depends on the choice of $\beta_x$.
\end{rem}
\begin{defn}
  \label{defn:FullHybridTrajectory}
  A \emph{full hybrid trajectory} through $x \in X$ is a map on $\dom(\beta_x) \cup \dom(\psi_x)$ that agrees with $\beta_x$ on $\dom(\beta_x)$ and with $\psi_x$ on $\dom(\psi_x)$, where $\beta_x$ is a backward hybrid trajectory through $x$ and $\psi_x$ is the hybrid trajectory of $x$.
\end{defn}
We define invariant sets and isolated invariant sets in the hybrid context.
\begin{defn}
  \label{defn:HybridInvariantSet}
  Let $\cH=(X,\varphi,G,r)$ be an HDS. A set $S$ is \emph{forward invariant} if $S = \Inv^+(S,\cH)$ where
  \[
    \Inv^+(S,\cH) \coloneqq \setdef{x \in X}{\psi([0,\infty),x) \subseteq S}.
  \]
  Similarly, a set $S$ is \emph{backward invariant} if $S = \Inv^-(S,\cH)$ where
  \[
    \Inv^-(S,\cH) \coloneqq \setdef{x \in X}{\beta((-\infty,0],x) \subseteq S \text{ for some backward hybrid trajectory } \beta_x}.
  \]
  A set is \emph{invariant} if it is both forward and backward invariant. The \emph{maximal invariant set} in $K \subseteq X$ is 
  \[ 
  \Inv(K,\cH) \coloneqq \Inv^+(K,\cH) \cap \Inv^-(K,\cH).
  \] 
  A compact set $K$ is an \emph{isolating neighborhood} under $\cH$ if $S = \Inv(K,\cH) \subseteq \Int(K)$. In that case, we say that $S$ is an \emph{isolated invariant set} for $\cH$.
\end{defn}

\section{Hybrid Suspension Semiflow}
\label{sec:suspension}

In this section we review the construction of the hybrid suspension semiflow introduced in \cite{Kvalheim2021}, which provides the foundation for defining the Suspension Conley Index. Then we establish the correspondence between invariant sets in the hybrid system and the suspension, and prove that the property of isolation is preserved under this correspondence.

First, recall that a \emph{strong deformation retract} of $U$ onto $G$ is a continuous map $H: U \times [0,1] \to U$ such that $H(x,0) = x$ for all $x \in U$, $H(x,1) \in G$ for all $x \in U$, and $H(g,s) = g$ for all $g \in G$ and $s \in [0,1]$. Figure~\ref{fig:DeformationRetract} exhibits two distinct strong deformation retractions of a neighborhood onto the guard, the first via straight lines and the second induced by the dynamics.
\begin{figure}[ht]
  \centering
  \tikzset{
    sheet/.pic={
      \draw[thick] (0,0) .. controls (1.8,-0.5) and (3.6,0.5) .. (5.4,0);
      \draw[thick] (1.3,3.0) .. controls (3.1,2.5) and (4.9,3.5) .. (6.7,3.0);
      \draw[thick] (0,0) -- (1.3,3.0);
      \draw[thick] (5.4,0) -- (6.7,3.0);
    },
    track/.style={gray!60!black},
    lead/.style={gray, very thin}
  }
  \def\neighborhood{
    \fill[gray!12] plot[domain=0:360, samples=100] ({2*cos(\x)}, {sin(\x)}) -- cycle;
    \draw[dashed] plot[domain=0:360, samples=100] ({2*cos(\x)}, {sin(\x)}) -- cycle;
  }
  \def\guard{
    \fill[gray!30] plot[domain=0:360, samples=100] ({cos(\x)}, {0.5*sin(\x)}) -- cycle;
    \draw[thick] plot[domain=0:360, samples=100] ({cos(\x)}, {0.5*sin(\x)}) -- cycle;
    \node at (0.55,0.1) {$G$};
    \node at ({1.5*cos(75)}, {0.75*sin(75)+0.42}) {$U$};
  }
  \begin{subfigure}{0.45\textwidth}
    \adjustbox{max width=\linewidth}{
      \begin{tikzpicture}[font=\small]
        \pic {sheet};
        \begin{scope}[shift={(3.4,1.5)}, rotate=15]
          \neighborhood
          \foreach \t in {0,30,...,150,210,240,...,330} {
            \draw[track] ({2*cos(\t)}, {sin(\t)}) -- ({cos(\t)}, {0.5*sin(\t)});
          }
          \guard
          \draw[thick] (-2,0) -- (-1,0);
          \fill (-1.75,0) circle (1.4pt) coordinate (x);
          \fill (-1.375,0) circle (1.4pt) coordinate (h);
          \fill (-1,0) circle (1.4pt) coordinate (e);
        \end{scope}
        \node[above=1pt] at (x) {$x$};
        \node[right=1pt] at (e) {$H(x,1)$};
      \end{tikzpicture}
    }
    \caption{Along segments}
  \end{subfigure}
  \begin{subfigure}{0.45\textwidth}
    \adjustbox{max width=\linewidth}{
      \begin{tikzpicture}[font=\small]
        \pic {sheet};
        \begin{scope}[shift={(3.4,1.5)}, rotate=15]
          \neighborhood
          \foreach \t in {0,75,255,300} {
            \draw[blue, midarrow=0.5] plot[domain=-0.15:1, samples=40, variable=\u] ({(2-\u)*cos(\t+60*\u)}, {(1-0.5*\u)*sin(\t+60*\u)});
          }
          \guard
          \draw[blue, thick, midarrow=0.85] plot[domain=-0.15:1, samples=40, variable=\u] ({(2-\u)*cos(165+60*\u)}, {(1-0.5*\u)*sin(165+60*\u)});
          \fill (-1.75,0) circle (1.4pt) coordinate (x);
          \fill ({1.375*cos(202.5)}, {0.6875*sin(202.5)}) circle (1.4pt) coordinate (h);
          \fill ({cos(225)}, {0.5*sin(225)}) circle (1.4pt) coordinate (e);
          \coordinate (p) at (2.2,0.2);
        \end{scope}
        \node[below=0pt] at (x) {$x$};
        \node[above=0pt, xshift=8pt] at (e) {$H(x,1)$};
        \node[blue] at (p) {$\varphi$};
      \end{tikzpicture}
    }
    \caption{Along trajectories of $\varphi$}
  \end{subfigure}
  \caption{Strong deformation retractions $H$ of a neighborhood $U$ of $G$ onto $G$. Each $x \in U$ moves along a path in $U$ from $H(x,0) = x$ to $H(x,1) \in G$, and every point of $G$ stays fixed. In (a) the paths are segments, while in (b) the paths are trajectories of $\varphi$, and $H(x,s) = \varphi(s\stoptime(x),x)$.}
  \label{fig:DeformationRetract}
\end{figure}
\begin{defn}[\cite{Kvalheim2021}]
  \label{defn:TrappingGuard}
  Let $\cH = (X,\varphi,G,r)$ be an HDS.
  We say that $\cH$ satisfies the \emph{Trapping Guard Condition} (TGC) if there exists an open neighborhood $U \subseteq X$ of $G$ such that the \emph{impact time} $\stoptime|_U$ is continuous on $U$ and the map $H(x,s) = \varphi(s \cdot \stoptime(x), x)$ is a strong deformation retract of $U$ onto $G$.
\end{defn}
\begin{rem}
  \label{rem:FlowExtension}
  In Definition~\ref{defn:TrappingGuard}, $\varphi$ is evaluated at $(\stoptime(x),x)$, and at $(0,g)$ for $g \in G$, which are not in the domain $\cU \subseteq \Rp \times (X \setminus G)$ of $\varphi$. However, let $\tilde\varphi$ be the map with $\tilde\varphi|_{\cU} = \varphi$, $\tilde\varphi(0,g) = g$ for $g \in G$, and $\tilde\varphi(\stoptime(x),x) = \lim_{t \to \stoptime(x)^-} \varphi(t,x)$ for each $x \in X \setminus G$ with $\stoptime(x) < \infty$ for which this limit exists. By Definition~\ref{defn:DynamicalSystem}(ii), this limit is in $G$. The map $H$ is defined through $\tilde\varphi$, so the TGC requires this limit to exist for every $x \in U$. From now on we denote $\tilde\varphi$ by $\varphi$.
\end{rem}
Since $\stoptime = 0$ on $G$, the TGC is equivalent to the existence of an open neighborhood $U$ of $G$ such that $\stoptime$ is finite and continuous on $U$, $\varphi(t,x)$ is defined and belongs to $U$ for every $x \in U$ and $t \in [0,\stoptime(x)]$, and $\varphi$ is continuous on $\setdef{(t,x) \in \Rp \times U}{t \leq \stoptime(x)}$~\cite{Kvalheim2021}, \cite[Proposition~2.8]{KaliesRivasWehbe2026}. Thus, under the TGC, every trajectory that starts in $U$ reaches $G$ in finite time without leaving $U$, and its impact time depends continuously on its initial point.

We highlight that continuous dynamics with $G = \emptyset$, and discrete-time dynamical systems, with $G = X$, satisfy the TGC trivially. Beyond that, two broad classes of hybrid systems satisfy the trapping guard condition~\cite{Kvalheim2021}. In the first, $X$ is a manifold with boundary and the guard is the boundary through which the flow leaves $X$, as in models of legged locomotion. In the second, the guard is the level set of a smooth function where the flow leaves a superlevel set, which includes mechanical systems with unilateral constraints that undergo impacts, such as hybrid Lagrangian systems.

\subsection{Construction of the Hybrid Suspension}
\label{subsec:SuspensionConstruction}

The hybrid suspension construction of~\cite{Kvalheim2021} proceeds in two steps (Figure~\ref{fig:hybridsuspension}). First, we attach $G\times[0,1]$ to $X$ by identifying $(g,0)$ with $g$ for each $g\in G$. We extend $\varphi$ along each interval $\{g\}\times[0,1]$ at unit speed and apply the reset $(g,1)\mapsto r(g)$ at its endpoint, obtaining the relaxed hybrid system. We then identify $(g,1)$ with $r(g)$ to obtain the hybrid suspension. Under the TGC, the induced dynamics form a continuous semiflow (Theorem~\ref{thm:hybridsemiflow}).

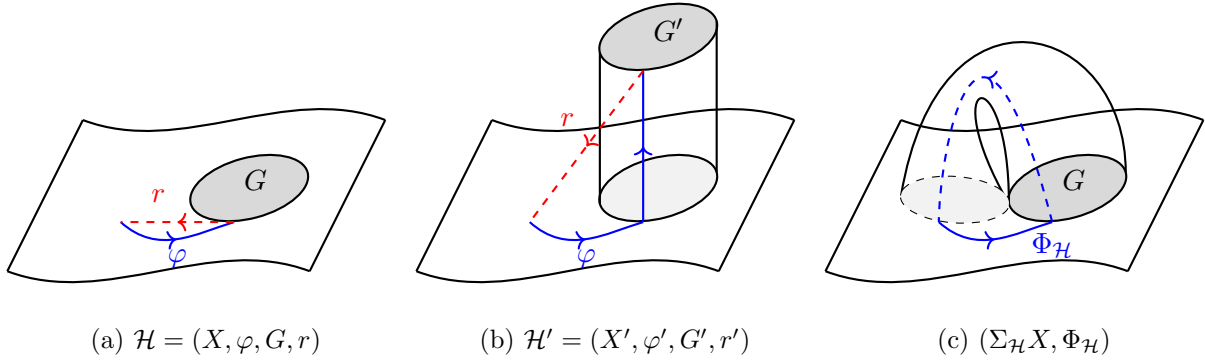
\begin{figure}[H]
  \centering
  \begin{subfigure}{0.32\textwidth}
    \adjustbox{max width=\linewidth}{
      \begin{tikzpicture}
        \coordinate (a) at (0,0);
        \coordinate (b) at (4,0);
        \coordinate (c) at (1,2);
        \coordinate (d) at (5,2);

        \coordinate (center) at (3.2,1.1);
        \coordinate (right_corner) at ($(center) + ({0.8*cos(15)}, {0.8*sin(15)})$);
        \coordinate (left_corner) at ($(center) + ({0.8*cos(15+180)}, {0.8*sin(15+180)})$);
        \coordinate (top_corner) at ($(center) + ({0.4*cos(15+90)}, {0.4*sin(15+90)})$);
        \coordinate (bottom_corner) at ($(center) + ({0.4*cos(15+270)}, {0.4*sin(15+270)})$);

        \draw[thick,black] (a) .. controls ($(a)!0.33!(b) + (0,-0.5)$) and ($(a)!0.67!(b) + (0,0.5)$) .. (b);
        \draw[thick,black] (c) .. controls ($(c)!0.33!(d) + (0,-0.5)$) and ($(c)!0.67!(d) + (0,0.5)$) .. (d);
        \draw[thick,black] (a) -- (c);
        \draw[thick,black] (b) -- (d);

        \fill[gray!30] (center) ellipse[x radius=0.8, y radius=0.4, rotate=15];
        \draw[thick,black] (center) ellipse[x radius=0.8, y radius=0.4, rotate=15];
        \node[right] at (3,1.2) {$G$};

        \draw[thick,midarrow=0.5,blue] (1.5,0.65) .. controls (2,0.2) and (2.5,0.5) .. (3,0.65) node[midway,below,blue] {$\varphi$};
        \draw[thick,dashed,red,midarrow=0.5] (3,0.65) -- (1.5,0.65);
        \node[red] at (2,1) {$r$};
    \end{tikzpicture}}
    \caption{$\cH = (X,\varphi,G,r)$}
  \end{subfigure}
  \begin{subfigure}{0.32\textwidth}
    \adjustbox{max width=\linewidth}{
      \begin{tikzpicture}
        \coordinate (a) at (0,0);
        \coordinate (b) at (4,0);
        \coordinate (c) at (1,2);
        \coordinate (d) at (5,2);

        \coordinate (center) at (3.2,1.1);
        \coordinate (right_corner) at ($(center) + ({0.8*cos(15)}, {0.8*sin(15)})$);
        \coordinate (left_corner) at ($(center) + ({0.8*cos(15+180)}, {0.8*sin(15+180)})$);
        \coordinate (top_corner) at ($(center) + ({0.4*cos(15+90)}, {0.4*sin(15+90)})$);
        \coordinate (bottom_corner) at ($(center) + ({0.4*cos(15+270)}, {0.4*sin(15+270)})$);

        \draw[thick,black] (a) .. controls ($(a)!0.33!(b) + (0,-0.5)$) and ($(a)!0.67!(b) + (0,0.5)$) .. (b);
        \draw[thick,black] (c) .. controls ($(c)!0.33!(d) + (0,-0.5)$) and ($(c)!0.67!(d) + (0,0.5)$) .. (d);
        \draw[thick,black] (a) -- (c);
        \draw[thick,black] (b) -- (d);

        \fill[gray!10] (center) ellipse[x radius=0.8, y radius=0.4, rotate=15];
        \draw[thick,black] (center) ellipse[x radius=0.8, y radius=0.4, rotate=15];
        \fill[gray!30] ($(center)+(0,2)$) ellipse[x radius=0.8, y radius=0.4, rotate=15];
        \draw[thick,black] ($(center)+(0,2)$) ellipse[x radius=0.8, y radius=0.4, rotate=15];

        \node[right] at (3,3.2) {$G'$};
        \draw[thick,black] (left_corner) -- ($(left_corner)+(0,2)$);
        \draw[thick,black] (right_corner) -- ($(right_corner)+(0,2)$);

        \draw[thick,midarrow=0.5,blue] (1.5,0.65) .. controls (2,0.2) and (2.5,0.5) .. (3,0.65) node[midway,below,blue] {$\varphi$};
        \draw[thick,midarrow=0.5,blue] (3,0.65) -- (3,2.65);
        \draw[thick,dashed,red,midarrow=0.5] (3,2.65) -- (1.5,0.65);
        \node[red] at (2,2) {$r'$};
    \end{tikzpicture}}
    \caption{$\cH' = (X',\varphi',G',r')$}
  \end{subfigure}
  \begin{subfigure}{0.32\textwidth}
    \adjustbox{max width=\linewidth}{
      \begin{tikzpicture}
        \coordinate (a) at (0,0);
        \coordinate (b) at (4,0);
        \coordinate (c) at (1,2);
        \coordinate (d) at (5,2);

        \coordinate (center) at (3.2,1.1);
        \coordinate (right_corner) at ($(center) + ({0.8*cos(15)}, {0.8*sin(15)})$);
        \coordinate (left_corner) at ($(center) + ({0.8*cos(15+180)}, {0.8*sin(15+180)})$);
        \coordinate (top_corner) at ($(center) + ({0.4*cos(15+90)}, {0.4*sin(15+90)})$);
        \coordinate (bottom_corner) at ($(center) + ({0.4*cos(15+270)}, {0.4*sin(15+270)})$);

        \draw[thick,black] (a) .. controls ($(a)!0.33!(b) + (0,-0.5)$) and ($(a)!0.67!(b) + (0,0.5)$) .. (b);
        \draw[thick,black] (c) .. controls ($(c)!0.33!(d) + (0,-0.5)$) and ($(c)!0.67!(d) + (0,0.5)$) .. (d);
        \draw[thick,black] (a) -- (c);
        \draw[thick,black] (b) -- (d);

        \fill[gray!30] (center) ellipse[x radius=0.8, y radius=0.4, rotate=15];
        \draw[thick,black] (center) ellipse[x radius=0.8, y radius=0.4, rotate=15];
        \node[right] at (3,1.2) {$G$};
        \draw[thick,black] (left_corner) .. controls ($(left_corner)+2*(0,1)$) and ($(left_corner)+2*(-0.5,0.86)$) .. (left_corner);
        \draw[thick,black] (right_corner) .. controls ($(right_corner)+2*(0,1)$) and ($(1.5,4)$) .. (1,1);

        \draw[thick,black,dashed] ($(1,1)!0.5!(left_corner)$) ellipse[x radius=0.716, y radius=0.3, rotate=-4.3];
        \fill[gray!10] ($(1,1)!0.5!(left_corner)$) ellipse[x radius=0.716, y radius=0.3, rotate=-4.3];

        \draw[thick,midarrow=0.5,blue] (1.5,0.65) .. controls (2,0.2) and (2.5,0.5) .. (3,0.65) node[right,below,blue] {$\Phi_\cH$};
        \draw[thick,dashed,blue,midarrow=0.5] (3,0.65) .. controls (2,4.5) and (1.5,1.5) .. (1.5,0.65) -- (1.5,0.65);
    \end{tikzpicture}}
    \caption{$(\susp X,\Phi_\cH)$}
  \end{subfigure}
  \caption{Hybrid suspension semiflow schematic.}
  \label{fig:hybridsuspension}
\end{figure}

\begin{defn}[\cite{Kvalheim2021}]
  \label{defn:RelaxedHybridSystem}
  Let $\cH = (X,\varphi,G,r)$ be an HDS satisfying the TGC. The \emph{relaxed hybrid system} $\cH' = (X',\varphi',G',r')$ is defined as follows:
  \begin{itemize}
    \item The space $X'$ is the quotient of the disjoint union $X \sqcup (G \times [0,1])$ by $g \sim (g,0)$ for $g \in G$, that is, the mapping cylinder of the inclusion $G \hookrightarrow X$~\cite{HatcherAT}. Let $\iota : X \hookrightarrow X'$ be the inclusion, so that $\iota(g) = (g,0)$ for $g \in G$.
    \item
      The local semiflow $\varphi'$ on $X' \setminus G'$ extends $\varphi$ via $\dom(\varphi') = \setdef{(t,z)}{t < \stoptime'(z)}$, where $\stoptime'(\iota(x)) = \stoptime(x) + 1$ for $x \in X$ and $\stoptime'(x,s) = 1-s$ for $(x,s) \in G \times (0,1]$, and, together with its extension of Remark~\ref{rem:FlowExtension}, it is defined by
      \[
        \varphi'(t,z) =
        \begin{cases}
          \iota(\varphi(t,x)) & \text{if } z=\iota(x),\, x \notin G,\, t < \stoptime(x),\\
          (\varphi(\stoptime(x),x), t-\stoptime(x)) & \text{if } z=\iota(x),\, x \notin G,\, \stoptime(x) \leq t \leq \stoptime'(z),\\
          (x,s+t) & \text{if } z=(x,s) \in G \times [0,1],\, s + t \leq 1.
        \end{cases}
      \]
    \item The guard set is $G' = G \times \{1\}$.
    \item The reset map $r' : G' \to X'$ is defined by $r'(x,1)=\iota(r(x))$.
  \end{itemize}
\end{defn}
Note that $\stoptime$ is continuous on $X$ and $X'$ is a compact metric space~\cite{Kvalheim2021}. Thus $\varphi'$ is continuous since it concatenates $\varphi$ with the gradient-like semiflow along the cylinder.
Now, by identifying the guard $G'$ with its image under $r'$, we obtain the suspension space and a continuous semiflow.
\begin{defn}[\cite{Kvalheim2021}]
  \label{defn:SuspensionSemiflow}
  Let $\cH$ be an HDS satisfying the TGC, and let $\cH'$ be the relaxed hybrid system. The \emph{hybrid suspension space} $\susp X$ is the quotient space $X'/\sim$, where $x' \sim r'(x')$ for $x' \in G'$.
  The \emph{hybrid suspension semiflow} $\Phi_\cH : \Rp \times \susp X \to \susp X$ is defined for $x' \in X'$ and $(t,n) \in \dom(\psi'_{x'})$, where $\psi'_{x'}$ is the hybrid trajectory of $x'$ in $\cH'$, by
  \[
    \Phi_\cH(t,[x']) = [\psi'_{x'}(t,n)].
  \]
\end{defn}
The suspension identifies $(g,0)$ with $g$ and $(g,1)$ with $r(g)$ for each $g\in G$, as shown in Figure~\ref{fig:hybridsuspension}. Since $X$ is a compact metric space, $G\subset X$ is closed, and $r$ is continuous, $\susp X$ is a compact metric space~\cite{Kvalheim2021}.
\[
  \susp X \cong \frac{X \sqcup (G \times [0,1]) } {(x,0)\sim x, (x,1)\sim r(x)}.
\]

\begin{thm}[\cite{Kvalheim2021}]
  \label{thm:hybridsemiflow}
  $\cH$ satisfies the TGC if and only if $(\susp X, \Phi_\cH)$ is a semiflow.
\end{thm}
Throughout the manuscript, we use the following maps relating these spaces. Let $\pi_0 : X' \to X$ be the projection map that collapses the cylinder and let $\pi: X' \to \susp X$ be the quotient map from Definition~\ref{defn:SuspensionSemiflow}. We denote by
\[
  \incl = \pi\circ\iota : X \to \susp X
\]
the inclusion into the suspension, and define
\[
  \susp: \cP(X) \to \cP(\susp X), \quad \susp(A) = \pi(\pi_0^{-1}(A))
\]
the correspondence that associates a subset of $X$ with its suspension in $\susp X$. We write simply $\susp(x)$ for singletons. This relationship is illustrated below.
\[
  \begin{tikzcd}
    {X'} && {\cP(X')} \\
    X & {\susp X} & {\cP(X)} & {\cP(\susp X)}
    \arrow["{{{{{\pi_0}}}}}", from=1-1, to=2-1]
    \arrow["\pi", from=1-1, to=2-2]
    \arrow["\pi"', from=1-3, to=2-4]
    \arrow["\iota", bend left=20, from=2-1, to=1-1]
    \arrow["{{{{\incl}}}}"', dashed, from=2-1, to=2-2]
    \arrow["{{\pi_0^{-1}}}", from=2-3, to=1-3]
    \arrow["{{{{{\susp}}}}}"', dashed, from=2-3, to=2-4]
\end{tikzcd}\]
We also use the following sets. For a closed set $A \subseteq X$, define
\begin{equation}
\label{eq:hatSet}
  \hA = \incl(A) \cup \pi\left((A \cap G) \times [0,1/2]\right) \cup \pi\left(r\inv(A) \times [1/2,1]\right).
\end{equation}
Note that $\susp A\setminus \hA = \pi(((A \cap G) \setminus r\inv(A)) \times (1/2,1])$, and since $A$ and $G$ are closed and $r$ is continuous, $\hA$ is compact. Moreover, the identities $\pi(g,0) = \incl(g)$ and $\pi(g,1) = \incl(r(g))$ for $g \in G$ imply that $\incl\inv(\hA) = A$.

\begin{ex}
  \label{ex:ConsecutiveResets}
  In this example, trajectories undergo different numbers of consecutive resets. Let $X = [0,3]$, $\varphi(t,x) = x+t$ for $x \in [0,1)$ and $t \in [0,1-x)$, $G = [1,3]$, and $r(x) = x-1$. The impact time $\stoptime(x) = \max\setof{1-x,0}$ is continuous on $X$, and $H(x,s) = \varphi(s \cdot \stoptime(x), x)=x+s\stoptime(x)$ is a strong deformation retraction of $X$ onto $G$, so $\cH$ satisfies the trapping guard condition with $U = X$. Note that the hybrid trajectory of $x \in [0,1)$ flows to $1$, resets to $r(1) = 0$, and repeats. For $k \in \setof{1,2,3}$ and $x \in [k,k+1)$, the points $x, r(x), \ldots, r^{k-1}(x)$ are in $G$ and $r^k(x) = x-k \in [0,1)$, so the trajectory of $x$ initializes with $k$ consecutive resets. The relaxed space $X'$ consists of $X$ together with $G \times [0,1]$ attached along $G \times \setof{0}$, and $r'(x,1) = \iota(x-1)$. In $\susp X$, the top of each segment $\setof{x} \times [0,1]$ is identified with $x-1$. Figure~\ref{fig:ConsecutiveResets}(c) shows $\susp X$ embedded in $\R^2$, with $\incl(X)$ as the segment $[0,3] \times \setof{0}$. The periodic orbit $\susp([0,1])$ of period $2$ is $[0,1]$ together with the segment over $1$.
\end{ex}
\begin{figure}[ht]
  \centering
  \begin{subfigure}[b]{0.32\textwidth}
    \centering
    \adjustbox{max width=\linewidth}{
      \begin{tikzpicture}[x=1.5cm]
        \useasboundingbox (-0.3,-1.6) rectangle (3.3,1.6);
        \draw[thick, blue, midarrow=0.6] (0,0) -- (1,0);
        \draw[very thick, red] (1,0) -- (3,0);
        \foreach \x in {0,1,2,3} {
          \fill (\x,0) circle (1.5pt);
          \node[below=2pt] at (\x,0) {$\x$};
        }
        \node[blue, below=12pt] at (0.5,0) {$\varphi$};
        \node[red, below=12pt] at (2,0) {$G$};
        \foreach \x in {1,2,3} {
          \draw[red, dashed, midarrow=0.55] (\x,0.05) to[out=100, in=80] (\x-1,0.05);
        }
        \node[red] at (0.5,1.05) {$r$};
    \end{tikzpicture}}
    \caption{$\cH = (X,\varphi,G,r)$}
  \end{subfigure}
  \hfill
  \begin{subfigure}[b]{0.32\textwidth}
    \centering
    \adjustbox{max width=\linewidth}{
      \begin{tikzpicture}[x=1.5cm, y=1.5cm]
        \useasboundingbox (-0.3,-0.6) rectangle (3.3,1.6);
        \fill[gray!12] (1,0) rectangle (3,1);
        \draw[thick] (1,0) -- (1,1);
        \draw[thick] (3,0) -- (3,1);
        \foreach \x in {1.5,2.5} {
          \draw[blue, midarrow=0.55] (\x,0) -- (\x,1);
        }
        \draw[thick, blue, midarrow=0.6] (0,0) -- (1,0);
        \draw[thick] (1,0) -- (3,0);
        \draw[very thick, red] (1,1) -- (3,1);
        \foreach \x in {0,1,2,3} {
          \fill (\x,0) circle (1.5pt);
          \node[below=2pt] at (\x,0) {$\x$};
        }
        \node[red, above=1pt] at (2,1) {$G'$};
        \foreach \x in {1,2,3} {
          \draw[red, dashed, midarrow=0.55] (\x,1) to[bend right=25] (\x-1,0.04);
        }
        \node[red] at (0.35,0.75) {$r'$};
    \end{tikzpicture}}
    \caption{$\cH' = (X',\varphi',G',r')$}
  \end{subfigure}
  \hfill
  \begin{subfigure}[b]{0.32\textwidth}
    \centering
    \adjustbox{max width=\linewidth}{
      \begin{tikzpicture}[scale=0.8]
        \fill[gray!12, even odd rule]
          plot[domain=0:1, samples=160, variable=\s] ({-0.5+(3.5-\s)*cos(360*\s)}, {(3.5-\s)*sin(360*\s)}) -- (3,0) -- cycle
          plot[domain=0:1, samples=160, variable=\s] ({-0.5+(1.5-\s)*cos(360*\s)}, {(1.5-\s)*sin(360*\s)}) -- (1,0) -- cycle;
        \foreach \x in {1.5,2,2.5,3} {
          \draw[blue, midarrow=0.5] plot[domain=0:1, samples=160, variable=\s] ({-0.5+(\x+0.5-\s)*cos(360*\s)}, {(\x+0.5-\s)*sin(360*\s)});
        }
        \draw[very thick] (1,0) -- (3,0);
        \draw[very thick, blue, midarrow=0.6] (0,0) -- (1,0);
        \draw[very thick, blue, midarrow=0.5] plot[domain=0:1, samples=160, variable=\s] ({-0.5+(1.5-\s)*cos(360*\s)}, {(1.5-\s)*sin(360*\s)});
        \foreach \x in {0,1,2,3} {
          \fill (\x,0) circle (2.2pt);
          \node[anchor=south west, inner sep=1pt] at ({\x+0.04},0.04) {\small $\x$};
        }
    \end{tikzpicture}}
    \caption{$(\susp X,\Phi_\cH)$}
  \end{subfigure}
  \caption{(a) The hybrid system on $X = [0,3]$ with $G = [1,3]$ and $r(x) = x-1$, with the flow in blue, the guard in red, and the resets dashed. (b) The relaxed system. (c) The suspension semiflow.}
  \label{fig:ConsecutiveResets}
\end{figure}

\subsection{Set-theoretic Properties}
In order to analyze invariant sets, we characterize the images and preimages of sets under the suspension and inclusion maps.
\begin{lem}
  \label{lem:InclInvSusp}
  For any $A \subset X$, $\incl\inv(\susp A) = A \cup r(A \cap G)$. 
  If $r(A\cap G)\subseteq A$, then $\incl\inv(\cl(\susp A))=\cl A$.
\end{lem}
\begin{proof}
  Note that $x \in \incl\inv(\susp A)$ if and only if $\pi(x,0) = \pi(x')$ for some $x' \in \pi_0\inv(A) = \iota(A) \cup ((A\cap G) \times (0,1])$. Since $\pi(G \times (0,1)) \cap \incl(X) = \emptyset$, either $x' \in \iota(A)$ or $x' \in (A \cap G) \times \{1\}$, which yields $x \in A$ or $x = r(g)$ for some $g \in A \cap G$. This proves the first part. For the second, recall that a closed continuous map $f$ satisfies $\cl(f(A))=f(\cl A)$. Since $X$ and $X'$ are compact metric spaces, the continuous maps $\incl$ and $\pi$ are closed. Thus, taking closures yields
  \begin{align*}
    \cl(\susp A)&=\cl(\incl(A))\cup\cl\left(\pi\left((A\cap G)\times[0,1]\right)\right) \\
    &=\incl(\cl A)\cup\pi\left(\cl(A\cap G)\times[0,1]\right)
  \end{align*}
  If $r(A\cap G)\subseteq A$, then
  \begin{align*}
    \incl\inv(\cl(\susp A))&=\cl A\cup r(\cl(A\cap G))=\cl A.
  \end{align*}
\end{proof}

\begin{lem}
  \label{lem:SuspInclInv}
  For any $B \subset \susp X$,
  \[
    \susp(\incl\inv(B)) = \left( B \cap \incl(X) \right) \cup \pi\left( \incl\inv\left( B \cap \incl(G) \right) \times (0,1] \right).
  \]
\end{lem}
\begin{proof}
  By definition, $x \in \incl\inv(B)$ if and only if $\incl(x) \in B$, that is, $\pi(x,0) \in B$.

  ($\subseteq$) Let $\tx \in \susp(\incl\inv(B))$. Then $\tx = \pi(x')$ for some $x' \in \pi_0\inv(\incl\inv(B))$. Since $\pi_0\inv(\incl\inv(B)) = \iota(\incl\inv(B)) \cup ((\incl\inv(B) \cap G) \times (0,1])$, either $x' = (x,0)$ for some $x \in \incl\inv(B)$, or $x' = (g,s)$ for some $g \in \incl\inv(B) \cap G$ and $s \in (0,1]$. If the former, then $\tx = \pi(x,0) = \incl(x) \in B \cap \incl(X)$.
  In the latter case, since $g \in \incl\inv(B)$, we have $\pi(g,0) = \incl(g) \in B \cap \incl(G)$, so $g \in \incl\inv(B \cap \incl(G))$. Thus $\tx = \pi(g,s) \in \pi(\incl\inv(B \cap \incl(G)) \times (0,1])$.

  ($\supseteq$) If $\tx \in B \cap \incl(X)$, then $\tx = \incl(x)$ for some $x \in X$ with $\incl(x) \in B$, so $x \in \incl\inv(B)$ and $\tx = \pi(x,0) \in \susp(\incl\inv(B))$. If $\tx \in \pi(\incl\inv(B \cap \incl(G)) \times (0,1])$, then $\tx = \pi(g,s)$ for some $g \in \incl\inv(B \cap \incl(G))$ and $s \in (0,1]$. Since $g \in \incl\inv(B \cap \incl(G))$, we have $\incl(g) \in B \cap \incl(G)$, which means $g \in \incl\inv(B) \cap G$, so $(g,s) \in \pi_0\inv(\incl\inv(B))$ and $\tx = \pi(g,s) \in \susp(\incl\inv(B))$.
\end{proof}

\subsection{Trajectories and Invariant Sets}

We establish the natural correspondence between hybrid trajectories in $X$, hybrid trajectories in $X'$, and continuous trajectories in $\susp X$.
\begin{lem}
  \label{lem:InclusionOrbit}
  For any $x \in X$ and $(t,n) \in \dom(\psi_x)$,
  \(
    \incl(\psi_x(t,n)) = \Phi_\cH(t+n,\incl(x)).
  \)
  Moreover, for $0 \leq n < J$ and $s \in [0,1]$,
  \(
    \Phi_\cH(t_{n+1}+n+s,\incl(x)) = \pi(g_n,s), \quad g_n = \psi_x(t_{n+1},n).
  \)
\end{lem}
\begin{proof}
  Recall that $t_0 = 0$ and $t_{n+1} - t_n = \stoptime(\psi_x(t_n,n))$, and similarly for $t'_n$. For each $n$, let $x_n = \psi_x(t_n,n)$, so that $x_0 = x$ and by Definition~\ref{defn:RelaxedHybridSystem}, $\stoptime'(\iota(x_n)) = \stoptime(x_n) + 1$ and
  \[
    \varphi'(t,\iota(x_n)) =
    \begin{cases}
      \iota(\varphi(t,x_n)) & \text{if } 0 \leq t \leq \stoptime(x_n), \\
      (g_n, t - \stoptime(x_n)) & \text{if } \stoptime(x_n) \leq t \leq \stoptime(x_n) + 1,
    \end{cases}
  \]
  where $g_n = \varphi(\stoptime(x_n),x_n) = \psi_x(t_{n+1},n)$. Since $r'(g_n,1) = \iota(r(g_n)) = \iota(x_{n+1})$, induction on $n$ yields $t'_n = t_n + n$, $\psi'_{\iota(x)}(t'_n,n) = \iota(x_n)$, $\psi'_{\iota(x)}(t+n,n) = \iota(\psi_x(t,n))$ for all $(t,n) \in \dom(\psi_x)$, and $\psi'_{\iota(x)}(t_{n+1}+n+s,n) = (g_n,s)$ for $0 \leq n < J$ and $s \in [0,1]$. By Definition~\ref{defn:SuspensionSemiflow},
  \[
    \Phi_\cH(t+n,\incl(x)) = \pi(\psi'_{\iota(x)}(t+n,n)) = \pi(\iota(\psi_x(t,n))) = \incl(\psi_x(t,n))
  \]
  and $\Phi_\cH(t_{n+1}+n+s,\incl(x)) = \pi(\psi'_{\iota(x)}(t_{n+1}+n+s,n)) = \pi(g_n,s)$.
\end{proof}

\begin{lem}
  \label{lem:SuspensionOrbit}
  For any $A \subset X$, $\displaystyle \susp(\psi([0,\infty),A)) = \Phi_\cH([0,\infty),\susp A)$.
\end{lem}
\begin{proof}
  Let $\tz \in \susp(\psi([0,\infty),A))$. Then $\tz \in \susp(x)$ for some $x = \psi_{x_0}(t,n)$ with $x_0 \in A$ and $(t,n) \in \dom(\psi_{x_0})$, and $\incl(x) = \Phi_\cH(t+n,\incl(x_0))$ by Lemma~\ref{lem:InclusionOrbit}. If $x \notin G$, then $\tz = \incl(x)$. If $x \in G$, then $\tz = \pi(x,s) = \Phi_\cH(s,\incl(x))$ for some $s \in [0,1]$. In both cases $\tz \in \Phi_\cH([0,\infty),\susp A)$.

  Conversely, let $\tz = \Phi_\cH(\tau,\ty)$ with $\ty \in \susp A$ and $\tau \geq 0$. If $\ty = \pi(g,s)$ with $g \in A \cap G$ and $s \in (0,1)$, then $\tz = \pi(g,s+\tau) \in \susp(g)$ for $\tau \leq 1-s$, and $\tz = \Phi_\cH(\tau-1+s,\incl(r(g)))$ with $r(g) = \psi_g(0,1)$ otherwise. Thus we may assume that $\ty = \incl(x_0)$ with $x_0 \in \psi([0,\infty),A)$. By Lemma~\ref{lem:InclusionOrbit}, either $\tau = t+n$ with $(t,n) \in \dom(\psi_{x_0})$ and $\tz = \incl(\psi_{x_0}(t,n))$, or $\tau = t_{n+1}+n+s$ with $s \in (0,1)$ and $\tz = \pi(g_n,s)$. In both cases $\tz \in \susp(\psi([0,\infty),x_0)) \subseteq \susp(\psi([0,\infty),A))$.
\end{proof}

\begin{lem}
  \label{lem:GuardInvariant}
  Let $(y,s) \in G\times [0,1)$. If $\pi(y,s) \in \tS$ for some invariant set $\tS$ under $\Phi_\cH$, then $\pi(\{y\} \times [0,1]) \subset \tS$.
\end{lem}
\begin{proof}
  Note that forward invariance yields $\pi(\{y\} \times [s,1]) \subset \tS$, while backward trajectories have the form $t \mapsto (y, s-t)$, so $\pi(\{y\} \times [0,s]) \subset \tS$ by backward invariance.
\end{proof}
We now establish the bijection between the invariant structures of $\cH$ and $\Phi_\cH$.
\begin{prop}
  \label{prop:HybridInvariantSet}
  If $S \subset X$ is an invariant set under $\cH$, then $\susp S $ is an invariant set under $\Phi_\cH$ and
  \(
    \incl\inv(\susp S) = S.
  \)
\end{prop}
\begin{proof}
  Since $S$ is invariant under $\cH$, $r(S \cap G) \subseteq S$. By Lemma~\ref{lem:InclInvSusp}, we obtain
  \[
    \incl\inv(\susp S) = S \cup r(S \cap G) = S.
  \]
  Now we verify invariance.
  Since $S$ is forward invariant, $\psi([0,\infty),S) = S$. Therefore, Lemma~\ref{lem:SuspensionOrbit} implies that $\Phi_\cH([0,\infty), \susp S) = \susp(\psi([0,\infty),S)) = \susp S$, so $\susp S$ is forward invariant. Backward invariance is analogous when considering the analogous suspension of backward hybrid trajectories.
\end{proof}

\begin{prop}
  \label{prop:SemiflowInvariantSet}
  If $\tS \subset \susp X$ is an invariant set under $\Phi_\cH$, then $\incl\inv(\tS)$ is an invariant set under $\cH$ and
  \(
    \susp(\incl\inv(\tS)) = \tS.
  \)
\end{prop}
\begin{proof}
  Let $S = \incl\inv(\tS)$.
  First, we verify invariance. For any $x \in S$, Lemma~\ref{lem:InclusionOrbit} implies that $\incl(\psi([0,\infty), x))$ is contained in $\Phi_\cH([0,\infty), \tS) = \tS$. Thus, $\psi([0,\infty), S) \subseteq \incl\inv(\tS) = S$. Backward invariance follows similarly by projecting backward trajectories of $\Phi_\cH$. To show $\susp S = \tS$, recall that Lemma~\ref{lem:SuspInclInv} yields
  \[
    \susp S = (\tS \cap \incl(X)) \cup \pi\left( \incl\inv(\tS \cap \incl(G)) \times (0,1] \right).
  \]
  Note that for $g \in \incl\inv(\tS \cap \incl(G))$, $\pi(g,1) = \incl(r(g)) \in \tS \cap \incl(X)$ by invariance, so it is enough to analyze $z = \pi(g,s)$ for $s \in (0,1)$. By Lemma~\ref{lem:GuardInvariant}, invariance of $\tS$ implies that $z \in \tS$ if and only if $\incl(g) \in \tS$, which is equivalent to $g \in \incl\inv(\tS \cap \incl(G))$. Thus, $\pi( \incl\inv(\tS \cap \incl(G)) \times (0,1))=\tS\setminus \incl(X)$. Substituting this identity back yields $\susp S = \tS$.
\end{proof}
\begin{lem}
  \label{lem:CompactIsolatedSet}
  If $\cH$ satisfies the TGC, then every isolated invariant set under $\cH$ is compact.
\end{lem}
\begin{proof}
  Let $K$ be an isolating neighborhood for $S$. We show that $\cl(S)$ is an invariant set contained in $K$. Hence $S = \cl(S)$ by maximality, and the result follows by compactness of $X$. Indeed, by Proposition~\ref{prop:HybridInvariantSet}, $\susp S$ is invariant and for each $t\geq 0$, we have that
  \[
    \Phi_\cH(t,\cl(\susp S))=\cl(\Phi_\cH(t,\susp S))=\cl(\susp S),
  \]
  since $\Phi_\cH(t,\cdot)$ is continuous. Taking preimages under $\Phi_\cH(1,\cdot)$, together with forward segments, yields backward trajectories in $\cl(\susp S)$, so $\cl(\susp S)$ is invariant under $\Phi_\cH$. By Lemma~\ref{lem:InclInvSusp} and Proposition~\ref{prop:SemiflowInvariantSet}, it follows that $\cl(S)=\incl\inv(\cl(\susp S))$ is invariant under $\cH$.
\end{proof}
\begin{lem}
  \label{lem:SuspendedIsolatingNeighborhood}
  For a compact set $K \subseteq X$ and $\hK$ given by \eqref{eq:hatSet}, $\Inv(\hK,\Phi_\cH) = \susp\Inv(K,\cH)$. Moreover, if $K$ is an isolating neighborhood under $\cH$, then $\hK$ is an isolating neighborhood under $\Phi_\cH$.
\end{lem}
\begin{proof}
  Let $S = \Inv(K,\cH)$. By Proposition~\ref{prop:HybridInvariantSet}, $\susp S$ is invariant, and $\susp S \subseteq \hK$ since $r(S \cap G) \subseteq S \subseteq K$, so $\susp S \subseteq \Inv(\hK,\Phi_\cH)$. Conversely, by Proposition~\ref{prop:SemiflowInvariantSet}, $A = \incl\inv(\Inv(\hK,\Phi_\cH))$ is invariant under $\cH$ and $\susp A = \Inv(\hK,\Phi_\cH)$. Since $A \subseteq \incl\inv(\hK) = K$, it follows that $A \subseteq S$, and therefore $\Inv(\hK,\Phi_\cH) \subseteq \susp S$.

  Now assume that $K$ is an isolating neighborhood, that is, that $S \subseteq \Int K$. Let $F = X \setminus \Int K$. Then $\susp X \setminus \hK \subseteq \hF$, where $\hF$ is compact, so $\cl(\susp X \setminus \hK) \subseteq \hF$. Since $S \cup r(S \cap G) \subseteq \Int K$, the set $\hF$ is disjoint from $\susp S$. Therefore $\susp S \subseteq \Int(\hK)$.
\end{proof}
\begin{thm}
  \label{thm:IsolatedSet}
  Let $S \subset X$ be an invariant set under $\cH$. Then $S$ is an isolated invariant set under $\cH$ if and only if $\susp S$ is an isolated invariant set under $\Phi_\cH$.
\end{thm}
\begin{proof}
  ($\Rightarrow$) Let $K$ be an isolating neighborhood for $S$ under $\cH$. Then $S = \Inv(K,\cH)$, and by Lemma~\ref{lem:SuspendedIsolatingNeighborhood}, $\hK$ is an isolating neighborhood under $\Phi_\cH$ with $\Inv(\hK,\Phi_\cH) = \susp S$.

  ($\Leftarrow$) Assume $\susp S$ is an isolated invariant set under $\Phi_\cH$. Let $\tB$ be an isolating neighborhood such that $\Inv(\tB, \Phi_\cH) = \susp S \subseteq \Int(\tB)$. We shall construct an isolating neighborhood $K$ of $S$. Since $\susp S$ is compact, $S=\incl\inv(\susp S)$ is closed in $X$, hence compact. First, define $V \subset G$ by
  \[
    V = \setdef{g \in G}{\pi(\setof{g} \times [0,1]) \subset \Int(\tB)}.
  \]
  By Lemma~\ref{lem:GuardInvariant}, if $g \in S \cap G$, then $\pi(\{g\} \times [0,1]) \subseteq \susp S \subseteq \Int(\tB)$. Thus, $S \cap G \subseteq V$, and $V$ is open in $G$ by compactness of $[0,1]$. Now, define $U = (\incl\inv(\Int(\tB))\setminus G) \cup V$ and note that $S \subseteq U$. Since $G\setminus V$ is closed in $X$ and $V\subseteq\incl\inv(\Int(\tB))$, $U$ is open in $X$. Let $U_0\subset X$ be an open neighborhood of $S$ such that $S \subset U_0 \subset \cl(U_0) \subset U$. Let $K=\cl(U_0)$. We claim that $\Inv(K, \cH) = S$. By construction, $S \subseteq \Inv(K, \cH)$. Conversely, $\susp \Inv(K, \cH)$ is an invariant set contained in $\susp U \subseteq \tB$, so $\susp \Inv(K,\cH) \subset \susp S$. Therefore, Proposition~\ref{prop:HybridInvariantSet} yields equality.
\end{proof}
The following corollary is a combination of Proposition~\ref{prop:SemiflowInvariantSet} together with Theorem~\ref{thm:IsolatedSet}.
\begin{cor}
  Let $\tS \subset \susp X$ be an invariant set under $\Phi_\cH$. Then $\tS$ is an isolated invariant set if and only if $\incl\inv(\tS)$ is an isolated invariant set under $\cH$.
\end{cor}

Having established the bijection between invariant sets and the preservation of isolation, we now define the Conley index for hybrid systems via the suspension construction.

\section{The Conley Index for Hybrid Systems}
\label{sec:conleyindex}

In this section, we define the Conley index for hybrid dynamical systems using the suspension semiflow constructed in Section~\ref{sec:suspension}. We also define hybrid index pairs, which allow for the computation of the index directly in the hybrid domain $X$ without explicit construction of the quotient space $\susp X$. For the remainder of this manuscript, let $\cH = (X,\varphi,G,r)$ be a hybrid dynamical system satisfying the Trapping Guard Condition (TGC).

\subsection{Suspension Conley Index}

The definition of the Suspension Conley Index relies on the correspondence between isolated invariant sets in $\cH$ and $\Phi_\cH$.
\begin{defn}
  \label{defn:SuspensionConleyIndex}
  Let $S \subset X$ be an isolated invariant set under $\cH$. By Theorem~\ref{thm:IsolatedSet}, $\susp S$ is an isolated invariant set for the suspension semiflow $\Phi_\cH$. The \emph{Suspension Conley Index} of $S$, denoted $\SCH_*(S)$, is defined as the homological Conley index of the corresponding isolated invariant set in the suspension semiflow:
  \[
    \SCH_*(S) \coloneqq CH_*(\susp S).
  \]
\end{defn}

Computing $\SCH_*(S)$ via this definition requires identifying the suspension of the invariant set in the suspension space $\susp X$. Traditionally, the theory does not keep track of invariant sets, but of isolating neighborhoods instead. We now describe an approach via hybrid index pairs that allows us to compute the index directly.

\subsection{Hybrid Index Pairs}

Hybrid index pairs are the fundamental computational tool for the Suspension Conley Index. They generalize classical index pairs to the setting of hybrid dynamical systems, accounting for both continuous dynamics and discrete resets.

\begin{defn}
  \label{defn:HybridIndexPair}
  Let $S \subset X$ be an isolated invariant set. A pair of compact sets $(N,L)$ with $L \subset N \subset X$ is a \emph{hybrid index pair} for $S$ if:
  \begin{enumerate}[(i)]
    \item $S = \Inv(\cl(N\setminus L),\cH) \subset \Int(N\setminus L)$.
    \item $L$ is positively invariant in $N$: if $x \in L$, then for all $(t,n) \in \dom(\psi_x)$ such that $\psi_x(s,m) \in N$ for all $(s,m) \leq (t,n)$, we have $\psi_x(t,n) \in L$.
    \item $L$ is an exit set for $N$: if $x \in N$ and there exists $(t,n) \in \dom(\psi_x)$ with $\psi_x(t,n) \notin N$, then there exists $(T,k) \in \dom(\psi_x)$ such that $\psi_x(s,m) \in N$ for all $(s,m) \leq (T,k)$ and $\psi_x(T,k) \in L$.
  \end{enumerate}
\end{defn}
Although intuitive, verifying these conditions via hybrid trajectories may not be straightforward. The following characterization uses $\varphi$ and $r$ explicitly.
\begin{prop}
  \label{prop:AlternativeHybridIndexPair}
  A pair of compact sets $(N,L)$ with $L \subset N$ is a hybrid index pair for $S$ if and only if Condition (i) of Definition~\ref{defn:HybridIndexPair} holds and:
  \begin{enumerate}[(a)]
    \item If $x \in L \cap G$ and $r(x) \in N$, then $r(x) \in L$.
    \item If $x \in L\setminus G$ and $\varphi([0,T],x)\subseteq N$ for some $T < \stoptime(x)$, then $\varphi([0,T],x)\subseteq L$.
    \item If $x \in N \cap G$ and $r(x) \notin N$, then $x \in L$.
    \item If $x \in N \setminus G$ and $\varphi([0,\stoptime(x)),x) \not\subseteq N$, then there exists $T < \stoptime(x)$ s.t. $\varphi([0,T],x) \subseteq N$ and $\varphi(T,x) \in L$.
  \end{enumerate}
\end{prop}
\begin{proof}
  We show that Definition~\ref{defn:HybridIndexPair}(ii) is equivalent to (a) and (b), and Definition~\ref{defn:HybridIndexPair}(iii) is equivalent to (c) and (d).

  \textit{Equivalence of (ii) and (a)--(b):}
  ($\Rightarrow$) Assume (ii). For (a), let $x \in L \cap G$ with $r(x) \in N$. We have $\psi_x(0,1) = r(x) \in N$ and $\psi_x(0,0) = x \in L$. Positive invariance then yields $r(x) \in L$. For (b), let $x \in L \setminus G$ and $\varphi([0,T],x) \subseteq N$. For $t \in [0,T]$, we have $\psi_x(t,0) = \varphi(t,x)$. Condition (ii) then implies $\varphi([0,T],x) \subseteq L$.

  ($\Leftarrow$) Assume (a) and (b). Let $x\in L$ and $(T,k)\in\dom(\psi_x)$ satisfy $\psi_x(s,m)\in N$ for all $(s,m)\leq(T,k)$, and let $J_n=I_n\cap[0,T]$ for $0\leq n\leq k$. If $\psi_x(t_n,n)\in L$, then condition (b) and continuity of $\gamma_n$ imply that  $\gamma_n(J_n)=\psi_x(J_n\times\setof{n})\subseteq L$. For $n<k$, this implies $\psi_x(t_{n+1},n)\in L\cap G$, and condition (a) yields
  \[
    \psi_x(t_{n+1},n+1)=r(\psi_x(t_{n+1},n))\in L.
  \]
  Since $\psi_x(t_0,0)=x\in L$, induction yields $\psi_x(T,k)\in L$, so (ii) holds.

  \textit{Equivalence of (iii) and (c)--(d):}
  ($\Rightarrow$) Assume (iii). For (c), let $x \in N \cap G$ with $r(x) \notin N$. We have $\psi_x(0,0) = x \in N$ and $\psi_x(0,1) = r(x) \notin N$. Condition (iii) guarantees an exit time $(T,k)<(0,1)$ with $\psi_x(T,k) \in L$, since the trajectory is in $N$ up to $(T,k)$. Thus $(T,k) = (0,0)$ and $x \in L$. For (d), let $x\in N\setminus G$ with $\varphi(t,x)\notin N$ for some $t<\stoptime(x)$. Condition (iii) gives $(T,m)<(t,0)$, since the trajectory is in $N$ up to $(T,m)$. Hence $m=0$ and $T<t$, with $\varphi([0,T],x)\subseteq N$ and $\varphi(T,x)\in L$.

  ($\Leftarrow$) Assume (c) and (d). Let $x \in N$ and suppose $\psi_x(t,n) \notin N$ for some $(t,n) \in \dom(\psi_x)$. We show that the trajectory exits through $L$. If the first exit occurs at a jump $(T,k)$, then $\psi_x(T,k) \in N \cap G$ while $r(\psi_x(T,k)) \notin N$. Condition (c) implies $\psi_x(T,k) \in L$. If the first exit occurs during the continuous segment indexed by $k$, let $y = \psi_x(t_k,k)$. Then $\varphi([0,\stoptime(y)),y) \not\subseteq N$ by continuity of $\gamma_k$ since $N$ is closed. Condition (d) yields $T < \stoptime(y)$ such that
  \[
    \varphi([0,T],y)\subseteq N \text{ and } \psi_x(t_k+T,k)=\varphi(T,y)\in L.
  \]
  Thus (iii) holds in both cases.
\end{proof}

\subsection{Suspension of Hybrid Index Pairs}

We now establish a correspondence between hybrid index pairs and classical index pairs in the suspension semiflow. Let $(N,L)$ be a hybrid index pair for $S$, and define
\begin{equation}
  \label{eq:EnlargedSuspensionPair}
  \tN = \susp N \cup \hN
  \text{ and }
  \tL = \susp L \cup \hL,
\end{equation}
where $\hN$ and $\hL$ are given by \eqref{eq:hatSet}. The sets $\tN$ and $\tL$ are compact, and $\tL \subseteq \tN$. By Lemma~\ref{lem:InclInvSusp} and $\incl\inv(\hA) = A$,
\[
  \incl\inv(\tN) = N \cup r(N \cap G)
  \text{ and }
  \incl\inv(\tL) = L \cup r(L \cap G).
\]
For $g \in G$ and $s \in (0,1)$, it follows that $\pi(g,s) \in \tN$ if and only if $g \in N$, or $s \geq 1/2$ and $r(g) \in N$. By Proposition~\ref{prop:AlternativeHybridIndexPair}(a) and (c),
\[
  r(L\cap G)\cap N\subseteq L
  \text{ and }
  r((N\setminus L)\cap G)\subseteq N.
\]
Hence $\incl\inv(\tL)\cap N=L$ and
\begin{equation}
\label{eq:EnlargedSuspensionDifference}
\tN\setminus\tL=\incl(N\setminus L)\cup\pi\left(((N\setminus L)\cap G)\times(0,1/2)\right)\cup\pi\left(r\inv(N\setminus L)\times[1/2,1)\right).
\end{equation}

  We break down the main result into three lemmas.

\begin{lem}
  \label{lem:SaturatedOpenSet}
  If $(N,L)$ is a hybrid index pair for $S$, then there exists an open set $W' \subset X'$ with $\pi_0\inv(S) \subset W'$ such that $\pi\inv(\pi(W')) = W'$ and $\pi(W')$ is open in $\susp X$ with $\pi(W') \subseteq \tN \setminus \tL$.
\end{lem}
\begin{proof}
  Define $U = \Int(N \setminus L) \setminus r\inv(L)$ and $V = r\inv(U)$. Since $r\inv(L)$ is closed in $X$ and $r$ is continuous, $U$ is open in $X$ and $V$ is open in $G$. Since $S \subset \Int(N \setminus L)$ and $r(S \cap G) \subseteq S$ by invariance of $S$, we have $S \cap r\inv(L) = \emptyset$, so $S \subset U$ and $S \cap G \subset V$. Define the set $W' \subset X'$ by
  \[
    W' = \iota(U) \cup \left((U \cap G) \times [0,1)\right) \cup \left(V \times (1/2,1]\right).
  \]
  Its preimage in $X \sqcup (G \times [0,1])$ is the union of the open sets $U$, $(U \cap G) \times [0,1)$, and $V \times (1/2,1]$, so $W'$ is open in $X'$. We have $\pi_0\inv(S) \subset W'$ since $\iota(S) \subset \iota(U)$, $(S \cap G) \times [0,1) \subset (U \cap G) \times [0,1)$, and $(S \cap G) \times \setof{1} \subset V \times (1/2,1]$.

  Now we verify that $\pi\inv(\pi(W')) = W'$. The inclusion $W' \subseteq \pi\inv(\pi(W'))$ is automatic, so it suffices to note that $(g,1) \in W'$ if and only if $g \in V$, which holds if and only if $\iota(r(g)) \in \iota(U) \subset W'$. Since $W'$ is open and $\pi\inv(\pi(W')) = W'$, the set $\pi(W')$ is open in $\susp X$.

  Finally, we show that $\pi(W') \subseteq \tN \setminus \tL$ using \eqref{eq:EnlargedSuspensionDifference}. If $x \in U$, then $x \in N \setminus L$, so $\incl(x) \in \tN \setminus \tL$. If $g \in U \cap G$, then $g \in N \setminus L$, and Proposition~\ref{prop:AlternativeHybridIndexPair}(c) yields $r(g) \in N$. Since $g \notin r\inv(L)$, we have $r(g) \in N \setminus L$, so $\pi(g,s) \in \tN \setminus \tL$ for all $s \in [0,1)$. If $g \in V$, then $r(g) \in U \subseteq N \setminus L$, so $\pi(g,s) \in \tN \setminus \tL$ for all $s \in (1/2,1]$.
\end{proof}

\begin{lem}
  \label{lem:PositiveInvarianceSuspension}
  If $(N,L)$ is a hybrid index pair for $S$ in $\cH$, then $\tL$ is positively invariant in $\tN$ under $\Phi_\cH$.
\end{lem}
\begin{proof}
  Let $\tx \in \tL$ and assume $\Phi_\cH([0,T], \tx) \subseteq \tN$ for some $T > 0$. We must show that $\Phi_\cH([0,T], \tx) \subseteq \tL$.

  Suppose that $\Phi_\cH(\tau,\tx) \notin \tL$ for some $\tau \in (0,T]$. Since $\tL$ is closed, the set $\setdef{t \in [0,\tau]}{\Phi_\cH(t,\tx) \in \tL}$ is compact and contains $0$, so it has a maximum $\tau_0 < \tau$. Let $\ty = \Phi_\cH(\tau_0,\tx) \in \tL$ and $\delta = \tau - \tau_0$. Then
  \[
    \Phi_\cH(t,\ty) \in \tN \setminus \tL \quad \forall t \in (0,\delta].
  \]
  Every point of $\susp X$ is $\pi(g,s)$ for some $g \in G$ and $s \in (0,1)$, or $\incl(y)$ for some $y \in X$. We show that each case contradicts the display above.

  If $\ty = \pi(g,s)$ with $g \in G$ and $s \in (0,1)$, then $\Phi_\cH(t,\ty) = \pi(g,s+t)$ for $t \in [0,1-s]$. Since $\ty \in \tL$, either $g \in L$, or $s \geq 1/2$ and $r(g) \in L$. In both cases $\Phi_\cH(t,\ty) \in \tL$ for all $t \in (0,\min\setof{\delta,1-s}]$.

  If $\ty = \incl(y)$ with $y \in G$, then $\Phi_\cH(t,\ty) = \pi(y,t)$ for $t \in [0,1]$. By \eqref{eq:EnlargedSuspensionDifference}, $\pi(y,t) \in \tN \setminus \tL$ for $t \in (0,\min\setof{\delta,1/2})$ implies $y \in N \setminus L$. But $y \in \incl\inv(\tL) \cap N = L$.

  If $\ty = \incl(y)$ with $y \in X \setminus G$, fix $\varepsilon \in (0,\min\setof{\delta,\stoptime(y)})$. Then $\Phi_\cH(t,\ty) = \incl(\varphi(t,y))$ for $t \in [0,\varepsilon]$, and \eqref{eq:EnlargedSuspensionDifference} yields $\varphi(t,y) \in N \setminus L$ for $t \in (0,\varepsilon]$. Since $N$ is closed, $y \in N$, and hence $y \in \incl\inv(\tL) \cap N = L$. Since $y \in L \setminus G$ and $\varphi([0,\varepsilon],y) \subseteq N$ with $\varepsilon < \stoptime(y)$, Proposition~\ref{prop:AlternativeHybridIndexPair}(b) yields $\varphi([0,\varepsilon],y) \subseteq L$, which contradicts $\varphi(\varepsilon,y) \notin L$.

  Therefore $\Phi_\cH([0,T],\tx) \subseteq \tL$.
\end{proof}

\begin{lem}
  \label{lem:ExitSetSuspension}
  If $(N,L)$ is a hybrid index pair for $S$ in $\cH$, then $\tL$ is an exit set for $\tN$ under $\Phi_\cH$.
\end{lem}
\begin{proof}
  Let $\tx \in \tN$ and suppose $\Phi_\cH([0,\infty), \tx) \not\subseteq \tN$. We must show that there exists $T \geq 0$ such that $\Phi_\cH(T, \tx) \in \tL$ and $\Phi_\cH([0,T], \tx) \subseteq \tN$.

  Define
  \[
    T = \inf\setdef{t \geq 0}{\Phi_\cH(t,\tx) \notin \tN}
    \text{ and }
    \ty = \Phi_\cH(T,\tx).
  \]
  Since $\tN$ is closed, $\Phi_\cH([0,T],\tx) \subseteq \tN$, and by the definition of $T$ there exist $\varepsilon_k > 0$ with $\varepsilon_k \to 0$ and $\Phi_\cH(\varepsilon_k,\ty) \notin \tN$ for all $k$. It remains to show that $\ty \in \tL$.

  If $\ty = \pi(g,s)$ with $g \in G$ and $s \in (0,1)$, then $\ty \in \tN$ yields $g \in N$, or $s \geq 1/2$ and $r(g) \in N$. In both cases $\Phi_\cH(t,\ty) = \pi(g,s+t) \in \tN$ for all $t \in [0,1-s]$, which contradicts $\Phi_\cH(\varepsilon_k,\ty) \notin \tN$ for $\varepsilon_k \leq 1-s$. Hence $\ty = \incl(y)$ for some $y \in X$, and three cases arise.

  If $y \notin N$, then $\ty \notin \tN \setminus \tL$ by \eqref{eq:EnlargedSuspensionDifference}, and since $\ty \in \tN$, it follows that $\ty \in \tL$.

  If $y \in N \cap G$, then $\Phi_\cH(t,\ty) = \pi(y,t) \in \susp N$ for all $t \in [0,1]$, which contradicts $\Phi_\cH(\varepsilon_k,\ty) \notin \tN$ for $\varepsilon_k \leq 1$.

  If $y \in N \setminus G$, then $\varepsilon_k < \stoptime(y)$ for all large $k$, and $\Phi_\cH(\varepsilon_k,\ty) = \incl(\varphi(\varepsilon_k,y)) \notin \tN$ yields $\varphi(\varepsilon_k,y) \notin N$. Thus $\varphi([0,\stoptime(y)),y) \not\subseteq N$, and Proposition~\ref{prop:AlternativeHybridIndexPair}(d) yields $T' < \stoptime(y)$ such that $\varphi([0,T'],y) \subseteq N$ and $\varphi(T',y) \in L$. Since $\varphi(\varepsilon_k,y) \notin N$, we have $T' < \varepsilon_k$ for all large $k$, so $T' = 0$ and $y \in L$. Therefore $\ty \in \susp L \subseteq \tL$.
\end{proof}

\begin{thm}
  \label{thm:SuspensionIndexPair}
  If $(N,L)$ is a hybrid index pair for $S$ under $\cH$, then $(\tN, \tL)$ given by \eqref{eq:EnlargedSuspensionPair} is an index pair for $\susp S$ under $\Phi_\cH$, and the inclusion $(\susp N, \susp L) \hookrightarrow (\tN, \tL)$ is a homotopy equivalence of pairs. In particular, $\SCH_*(S) \cong H_*(\susp N, \susp L)$.
\end{thm}
\begin{proof}
  We show that $(\tN, \tL)$ satisfies the three conditions of Definition~\ref{defn:IndexPair}. The sets $\tN$ and $\tL$ are compact, with $\tL \subseteq \tN$.

  We first show that $\susp S \subset \Int(\tN \setminus \tL)$. By Lemma~\ref{lem:SaturatedOpenSet}, there exists an open set $W' \subset X'$ such that $\pi_0\inv(S) \subset W'$, the set $\tW = \pi(W')$ is open in $\susp X$, and $\tW \subseteq \tN \setminus \tL$. Since $\susp S = \pi(\pi_0\inv(S)) \subset \tW$, we have $\susp S \subset \Int(\tN \setminus \tL)$.

  Now, we show that $\susp S = \Inv(\cl(\tN \setminus \tL), \Phi_\cH)$. Since $\susp S$ is invariant by Proposition~\ref{prop:HybridInvariantSet} and contained in $\tN \setminus \tL$, we have $\susp S \subseteq \Inv(\cl(\tN \setminus \tL), \Phi_\cH)$. Conversely, let $K = \cl(N \setminus L)$. Since $\hK$, given by \eqref{eq:hatSet}, is compact and contains $\tN \setminus \tL$ by \eqref{eq:EnlargedSuspensionDifference}, Lemma~\ref{lem:SuspendedIsolatingNeighborhood} yields
  \[
    \Inv(\cl(\tN \setminus \tL), \Phi_\cH) \subseteq \Inv(\hK, \Phi_\cH) = \susp\Inv(K, \cH) = \susp S.
  \]

  Since conditions (ii) and (iii) follow from Lemma~\ref{lem:PositiveInvarianceSuspension} and Lemma~\ref{lem:ExitSetSuspension}, $(\tN, \tL)$ is an index pair for $\susp S$.

  It remains to show that the inclusion $(\susp N, \susp L) \hookrightarrow (\tN, \tL)$ is a homotopy equivalence of pairs. For $u \in [0,1]$, $x \in X$, $g \in G$, and $s \in [0,1]$, define $h : [0,1] \times \susp X \to \susp X$ by
  \[
    h(u,\incl(x)) = \incl(x)
    \text{ and }
    h(u,\pi(g,s)) = \pi\left(g,\min\setof{(1+u)s,1}\right).
  \]
  Since $\min\setof{(1+u)s,1} = s$ for $s \in \setof{0,1}$, the two formulas agree on $\pi(g,0) = \incl(g)$ and $\pi(g,1) = \incl(r(g))$. Since $[0,1]$ is compact, the product of the identity of $[0,1]$ with the quotient map $X \sqcup (G \times [0,1]) \to \susp X$ is a quotient map, so $h$ is continuous. For each $u \in [0,1]$, the map $h(u,\cdot)$ sends each of $\susp N$, $\susp L$, $\tN$, and $\tL$ into itself, $h(0,\cdot)$ is the identity, and
  \[
    h(1,\tN) \subseteq \susp N
    \text{ and }
    h(1,\tL) \subseteq \susp L.
  \]
  Thus $h(1,\cdot) : (\tN, \tL) \to (\susp N, \susp L)$ is a homotopy inverse of the inclusion, with both homotopies given by restrictions of $h$.

  By the definition of the Conley index and homotopy invariance of relative homology,
  \[
    \SCH_*(S)\cong H_*(\tN,\tL)\cong H_*(\susp N,\susp L).
  \]
\end{proof}

Theorem~\ref{thm:SuspensionIndexPair} allows us to compute the Suspension Conley Index using familiar techniques from the hybrid domain, without explicitly constructing the quotient space.

\subsection{Continuation and Robustness}

One of the properties of the Conley index is its stability under perturbations. While the notion of a neighborhood of dynamical systems is well-established \cite{Alex2023}, it is lacking in the context of hybrid systems. Even in the suspension semiflow setting, it is easy to verify that small changes in $r$ might significantly alter the topological structure of $\susp X$. In what follows, we describe a concept of robustness for the Suspension Conley Index under continuous deformations of the local semiflow, while keeping the rest of the hybrid structure (state space, guard set, and reset map) fixed. Throughout, let $\Lambda=[-1,1]$.

\begin{defn}
  \label{defn:ContinuousFamily}
  A \emph{continuous family of hybrid dynamical systems} $\setof{\cH_\lambda}_{\lambda \in \Lambda}$ consists of hybrid systems $\cH_\lambda=(X,\varphi_\lambda,G,r_\lambda)$ where each $\varphi_\lambda : \cU_\lambda \subseteq \Rp \times (X \setminus G) \to X \setminus G$ is a local semiflow with domain $\cU_\lambda$, such that
  \begin{align*}
    \Phi : \cD &\to X, \quad \Phi(t,x,\lambda) = \varphi_\lambda(t,x), \\
    R : G \times \Lambda &\to X, \quad R(g,\lambda) = r_\lambda(g),
  \end{align*}
  are continuous, with $\cD = \setdef{(t,x,\lambda) \in \Rp \times (X \setminus G) \times \Lambda}{(t,x) \in \cU_\lambda}$ open in $\Rp \times (X \setminus G) \times \Lambda$.
\end{defn}
\begin{defn}
  \label{defn:UniformTGC}
  A continuous family of HDS $\{\cH_\lambda\}_{\lambda \in \Lambda}$ satisfies the \emph{Uniform Trapping Guard Condition} (UTGC) if there exists an open neighborhood $U \subseteq X$ of $G$ such that:
  \begin{enumerate}[(i)]
    \item The map $\Lambda \times U \to \Rp$ given by $(\lambda, x) \mapsto \stoptime_\lambda(x)$ is continuous.
    \item The map $H: \Lambda \times U \times [0,1] \to U$ given by $H(\lambda, x, t) = \varphi_\lambda(t \cdot \stoptime_\lambda(x), x)$ is continuous on $\Lambda \times U \times [0,1]$ and a strong deformation retract from $U$ to $G$ for each $\lambda$.
  \end{enumerate}
\end{defn}
\begin{defn}
  \label{defn:FlowFamily}
  A continuous family $\{\cH_\lambda\}_{\lambda \in \Lambda}$ is \emph{flow-dependent} if the reset map $r$ is independent of $\lambda$, that is, $\cH_\lambda = (X, \varphi_\lambda, G, r)$ for all $\lambda \in \Lambda$.
\end{defn}
\begin{rem}
  For flow-dependent families, the suspension space
  \[
    \Sigma_{\cH_\lambda} X = \frac{X \sqcup (G \times [0,1])}{(g,0) \sim g,\, (g,1) \sim r(g)}
  \]
  is independent of $\lambda$, so we write $\susp X$ for the suspension space. If the family also satisfies the UTGC, then the suspension semiflow $\Phi_{\cH_\lambda}$ varies continuously with $\lambda$ on the fixed space $\susp X$ by Proposition~\ref{prop:FlowDependentContinuity}, allowing classical continuation theorems for isolated invariant sets to apply directly.
\end{rem}
\begin{prop}
  \label{prop:FlowDependentContinuity}
  If $\{\cH_\lambda\}_{\lambda \in \Lambda}$ is a flow-dependent family satisfying the UTGC, then $\{\Phi_{\cH_\lambda}\}_{\lambda \in \Lambda}$ is a continuous family of semiflows on $\susp X$. That is, the map $(t, z, \lambda) \mapsto \Phi_{\cH_\lambda}(t,z)$ is continuous on $\Rp \times \susp X \times \Lambda$.
\end{prop}
\begin{proof}
  Consider the hybrid system $\cH$ on the product space $X \times \Lambda$ defined by
  \[
    \cH = (X \times \Lambda, \Phi, G \times \Lambda, R),
  \]
  where $\Phi(t, (x,\lambda)) = (\varphi_\lambda(t,x), \lambda)$ and $R(g,\lambda) = (r(g), \lambda)$.
  Since the family $\setof{\cH_\lambda}$ satisfies the UTGC, $\cH$ satisfies the Trapping Guard Condition on $X \times \Lambda$.
  By Theorem~\ref{thm:hybridsemiflow}, the suspension semiflow $\Phi_{\cH}$ is well-defined and continuous on $\susp(X \times \Lambda)$.

  Since the family is flow-dependent, that is, $r$ is fixed, there is a homeomorphism $h: \susp(X \times \Lambda) \to (\susp X) \times \Lambda$.
  The semiflow $\Phi_\cH$ on $\susp(X \times \Lambda)$ is conjugate via $h$ to the map $((z, \lambda), t) \mapsto (\Phi_{\cH_\lambda}(t,z), \lambda)$.
  Therefore, continuity of $\Phi_\cH$ implies the continuity of this map with respect to $(t, z, \lambda)$.
  Thus, $\setof{\Phi_{\cH_\lambda}}_{\lambda \in \Lambda}$ is a continuous family of semiflows on $\susp X$.
\end{proof}
\begin{thm}
  \label{thm:Robustness}
  Let $\setof{\cH_\lambda}_{\lambda \in \Lambda}$ be a flow-dependent family satisfying the UTGC.
  If $N$ is an isolating neighborhood in $\cH_0$, then there exists $\varepsilon > 0$ such that $N$ is an isolating neighborhood for $\cH_\lambda$ for all $\lambda \in (-\varepsilon, \varepsilon)$.
\end{thm}
\begin{proof}
  We prove it by contradiction. Since $r$ is fixed, $\hN$ in \eqref{eq:hatSet} does not depend on $\lambda$. Suppose that for each $k \in \N$, there exist $\lambda_k \in \Lambda$ with $|\lambda_k| < 1/k$ and $x_k \in \partial N \cap \Inv(N, \cH_{\lambda_k})$. By Lemma~\ref{lem:SuspendedIsolatingNeighborhood}, $\incl(x_k) \in \susp\Inv(N,\cH_{\lambda_k}) = \Inv(\hN,\Phi_{\cH_{\lambda_k}})$, so there exists a full trajectory $\gamma_k : \R \to \hN$ of $\Phi_{\cH_{\lambda_k}}$ with $\gamma_k(0) = \incl(x_k)$. Since $\partial N$ is compact, we may pass to a subsequence such that $x_k \to x^* \in \partial N$.

  We show that $\incl(x^*) \in \Inv(\hN, \Phi_{\cH_0})$. For $t \geq 0$, Proposition~\ref{prop:FlowDependentContinuity} yields
  \[
    \Phi_{\cH_0}(t, \incl(x^*)) = \lim_{k \to \infty} \Phi_{\cH_{\lambda_k}}(t, \incl(x_k)) = \lim_{k \to \infty} \gamma_k(t) \in \hN.
  \]
  For $m \in \N$, $\gamma_k(-m)\in\hN$, so a subsequence $(k_j)$ yields $w_m = \lim_{j \to \infty} \gamma_{k_j}(-m)$ for all $m$. By Proposition~\ref{prop:FlowDependentContinuity}, $w_0 = \incl(x^*)$, $\Phi_{\cH_0}(1,w_{m+1}) = w_m$, and $\Phi_{\cH_0}([0,1],w_{m+1}) \subseteq \hN$. Thus $\incl(x^*)$ admits a backward trajectory in $\hN$, and $\incl(x^*) \in \Inv(\hN, \Phi_{\cH_0})$.

  By Lemma~\ref{lem:SuspendedIsolatingNeighborhood} and Proposition~\ref{prop:HybridInvariantSet},
  \[
    x^* \in \incl\inv(\Inv(\hN, \Phi_{\cH_0})) = \incl\inv(\susp\Inv(N, \cH_0)) = \Inv(N, \cH_0).
  \]
  Thus $x^* \in \partial N \cap \Inv(N, \cH_0)$, contradicting $\Inv(N, \cH_0) \subseteq \Int N$.
\end{proof}
\begin{defn}
  \label{defn:Continuation}
  Let $\setof{\cH_\lambda}_{\lambda \in \Lambda}$ be a flow-dependent family satisfying the UTGC. For a compact set $N \subset X$ and $S_\lambda = \Inv(N,\cH_\lambda)$, we say $S_{\lambda_0}$ and $S_{\lambda_1}$ with $\lambda_0 \leq \lambda_1$ are \emph{related by continuation} if $N$ is an isolating neighborhood for $\cH_\lambda$ for all $\lambda \in [\lambda_0, \lambda_1]$.
\end{defn}

\begin{thm}
  \label{thm:Continuation}
  Let $\setof{\cH_\lambda}_{\lambda \in \Lambda}$ be a flow-dependent continuous family satisfying the UTGC. If $S_{\lambda_0}$ and $S_{\lambda_1}$ are related by continuation, then $\SCH_*(S_{\lambda_0}) \cong \SCH_*(S_{\lambda_1})$.
\end{thm}
\begin{proof}
  By the definition of continuation, $N$ is an isolating neighborhood for $\cH_\lambda$ for all $\lambda \in [\lambda_0, \lambda_1]$, and $S_\lambda = \Inv(N, \cH_\lambda)$. Since the family is flow-dependent, the reset map $r$ is independent of $\lambda$, so the suspension space $\susp X$ is a fixed compact metrizable space for all $\lambda$.

  Since $r$ is fixed, $\hN$ in \eqref{eq:hatSet} does not depend on $\lambda$, and by Lemma~\ref{lem:SuspendedIsolatingNeighborhood} it is an isolating neighborhood under $\Phi_{\cH_\lambda}$ with $\Inv(\hN,\Phi_{\cH_\lambda}) = \susp S_\lambda$ for each $\lambda \in [\lambda_0,\lambda_1]$.

  By Proposition~\ref{prop:FlowDependentContinuity}, the map $(\lambda, z, t) \mapsto \Phi_{\cH_\lambda}(t,z)$ is continuous on $[\lambda_0, \lambda_1] \times \susp X \times \Rp$. Since $[\lambda_0, \lambda_1]$ is connected and $\hN$ isolates $\susp S_\lambda$ for each $\lambda$, the classical
  continuation theorem for semiflows \cite{Handbook2002,Rybakowski1987} yields that the Conley index $CH_*(\susp S_\lambda)$ is constant for $\lambda \in [\lambda_0, \lambda_1]$. Therefore $CH_*(\susp S_{\lambda_0}) \cong CH_*(\susp S_{\lambda_1})$, and by Definition~\ref{defn:SuspensionConleyIndex}, $\SCH_*(S_{\lambda_0}) \cong \SCH_*(S_{\lambda_1})$.
\end{proof}

\begin{cor}
  Let $\setof{\cH_\lambda}_{\lambda \in \Lambda}$ be a flow-dependent family satisfying the UTGC. Under the assumptions of Theorem~\ref{thm:Robustness}, $\SCH_*(S_\lambda)$ is locally constant.
\end{cor}
\begin{proof}
  By Theorem~\ref{thm:Robustness}, there exists $\varepsilon > 0$ such that $N$ isolates $S_\lambda$ for all $\lambda \in (-\varepsilon, \varepsilon)$. Thus any $S_\lambda$ with $|\lambda| < \varepsilon$ is related by continuation to $S_0$, and by Theorem~\ref{thm:Continuation}, $\SCH_*(S_\lambda) \cong \SCH_*(S_0)$.
\end{proof}

\subsection{Reconstruction Theorems}

By leveraging classical results from Conley index theory through the suspension semiflow, we can obtain existence results for fixed points and periodic orbits in hybrid systems.

\begin{thm}
  \label{thm:HybridWazewski}
  If $\SCH_*(S) \neq 0$, then $S \neq \emptyset$.
\end{thm}
\begin{proof}
  Assume $\SCH_*(S) \neq 0$. By definition, $CH_*(\susp S) \neq 0$, so the classical Wa\.zewski property yields $\susp S \neq \emptyset$. Let $\tz \in \susp S$. If $\tz \in \incl(X)$, then $\incl\inv(\tz) \neq \emptyset$. If $\tz \in \pi(G \times (0,1))$, Lemma~\ref{lem:GuardInvariant} gives $\pi(\{g\} \times [0,1]) \subset \susp S$ for some $g \in G$, so $\incl(g) \in \susp S \cap \incl(X)$. In either case, $S = \incl\inv(\susp S) \neq \emptyset$.
\end{proof}

\begin{thm}
  \label{thm:FixedPoint}
  Let $S$ be an isolated invariant set. If, for some $n \in \Zp$,
  \[
    \SCH_k(S) \cong
    \begin{cases}
      \F & \text{if } k=n, \\
      0 & \text{otherwise},
    \end{cases}
  \]
  then $S$ contains a fixed point $p$ of $\cH$. Moreover, $p \notin G$.
\end{thm}
\begin{proof}
  Given that $\SCH_*(S) = CH_*(\susp S)$, it follows from Theorem~\ref{thm:McCord}(i) that $\susp S$ contains a fixed point. Note that if $\tq$ is a fixed point of $\Phi_\cH$, then $\tq \notin \pi(G \times (0,1))$ since $\Phi_\cH$ is gradient-like along the cylinder. By Lemma~\ref{lem:GuardInvariant}, it follows that $\incl\inv(\tq) \cap G = \emptyset$, and $\incl\inv(\tq) = \setof{p}$ for some $p \in X \setminus G$. By Lemma~\ref{lem:SuspensionOrbit}, $\psi([0,\infty),p)=\setof{p}$.
\end{proof}

\begin{defn}
  \label{defn:HybridPeriodic}
  Let $\cH=(X,\varphi,G,r)$ be an HDS and let $x \in X$. The point $x$ is a \emph{hybrid periodic point} of $\cH$ if there exists $(T,m) \in \dom(\psi_x)$ with $T+m>0$ such that $\psi_x(t+T,k+m)=\psi_x(t,k)$ for every $(t,k)\in\dom(\psi_x)$. The image $\psi_x(\dom(\psi_x))$ is a \emph{hybrid periodic orbit} of $\cH$. The orbit is \emph{purely continuous} if $m=0$ and \emph{purely discrete} if $T=0$.
\end{defn}

\begin{thm}
  \label{thm:PeriodicOrbit}
  Let $S$ be an isolated invariant set such that $\stoptime(x) < \infty$ for every $x \in S$. If, for some $n \in \Zp$,
  \[
    \SCH_k(S) \cong
    \begin{cases}
      \F & \text{if } k=n,n+1, \\
      0 & \text{otherwise},
    \end{cases}
  \]
  then $S$ contains a hybrid periodic orbit of $\cH$.
\end{thm}
\begin{proof}
  Under the hypothesis, $\SCH_*(S) = CH_*(\susp S)$ satisfies the conditions of Theorem~\ref{thm:McCord}(ii). Since $\stoptime < \infty$ on $S$, the forward orbit of every point of $\susp S$ intersects $\Xi = \pi(G \times \setof{1/2})$ at a positive time. Since $\Xi$ is a local section, the set of points with this property is open, so it contains an isolating neighborhood $N$ of $\susp S$, and $\Xi$ is a Poincaré section for $N$. Therefore, $\susp S$ contains a periodic orbit. Let $x \in X$ be such that $\susp(x)$ intersects the periodic orbit. By Lemma~\ref{lem:SuspensionOrbit}, the orbit is generated by $\susp(\psi([0,\infty),x)) = \Phi_\cH([0,\infty),\incl(x))$, so $x$ is a hybrid periodic point of $\cH$. If $\psi([0,\infty),x) \cap (X \setminus G) = \emptyset$, then $x \in G$ and
  \[
    \psi([0,\infty),x) = \setof{x,r(x),r^2(x),\ldots,r^{m-1}(x)},
  \]
  where $m$ is the period of the orbit. Therefore the orbit is purely discrete.
\end{proof}

If $S\cap G=\emptyset$, then the dynamics on $S$ are purely continuous.

\begin{cor}
  \label{cor:PeriodicOrbitI}
  Under the hypothesis of Theorem~\ref{thm:PeriodicOrbit}, if $r$ has no periodic points, then $S$ contains a hybrid periodic orbit that is not purely discrete. In particular, if $r(G) \cap G = \emptyset$, then $S$ contains a hybrid periodic orbit that is not purely discrete.
\end{cor}

\subsection{Hybrid Morse Representations}

The correspondence between invariant sets in $\cH$ and $\Phi_\cH$ extends to Morse representations. Given a Morse representation of an isolated invariant set in the suspension semiflow, we obtain a corresponding representation in the hybrid system.

\begin{defn}
  \label{defn:HybridMorseRepresentation}
  Let $S \subset X$ be an isolated invariant set under $\cH$. A \emph{hybrid Morse representation} of $S$ is a finite collection $\sM = \{M_1, \ldots, M_n\}$ of mutually disjoint, nonempty, compact invariant subsets of $S$, together with a strict partial order $<$ on $\sM$, such that $\{\susp M_1, \ldots, \susp M_n\}$ with the induced order is a Morse representation of $\susp S$ under $\Phi_\cH$.
\end{defn}
\begin{prop}
  \label{prop:MorseCorrespondence}
  Let $S \subset X$ be an isolated invariant set under $\cH$.
  \begin{enumerate}[(i)]
    \item If $\{\tM_1, \ldots, \tM_n\}$ is a Morse representation of $\susp S$ under $\Phi_\cH$, then the sets $M_i = \incl\inv(\tM_i)$ form a hybrid Morse representation of $S$.
    \item If $\{M_1, \ldots, M_n\}$ is a hybrid Morse representation of $S$, then each $M_i$ is an isolated invariant set under $\cH$ with $\SCH_*(M_i) = CH_*(\susp M_i)$.
  \end{enumerate}
\end{prop}
\begin{proof}
  For (i), by Proposition~\ref{prop:SemiflowInvariantSet}, each $M_i = \incl\inv(\tM_i)$ is invariant under $\cH$ with $\susp M_i = \tM_i$. The sets $M_i$ are mutually disjoint since $\incl$ is injective and the $\tM_i$ are disjoint.
  Since $\incl$ is continuous and $\tM_i$ is compact, each $M_i$ is closed in $X$, hence compact.
  The partial order on $\{M_i\}$ is inherited from $\{\tM_i\}$.

  For (ii), since $\{\susp M_1, \ldots, \susp M_n\}$ is a Morse representation of $\susp S$, each $\susp M_i$ is an isolated invariant set under $\Phi_\cH$. By Theorem~\ref{thm:IsolatedSet}, each $M_i$ is isolated under $\cH$. The index formula follows from Definition~\ref{defn:SuspensionConleyIndex}.
\end{proof}

\begin{rem}
  \label{rem:HybridMorseOrder}
  By Lemma~\ref{lem:SuspensionOrbit} and Proposition~\ref{prop:SemiflowInvariantSet}, full hybrid trajectories of $\cH$ on $S$ correspond, via $\incl$ and $\susp$, to full trajectories of $\Phi_\cH$ on $\susp S$.
  Thus, if $i\neq j$ and there exists a full hybrid trajectory in $S$ whose suspension in $\susp S$ has $\alpha$-limit in $\susp M_j$ and $\omega$-limit in $\susp M_i$, then $M_i<M_j$.
  A complete and intrinsic characterization in terms of hybrid limit sets, attractor-repeller pairs, and connection matrices is the subject of future work.
\end{rem}

\section{Examples}
\label{sec:examples}

In this section, we illustrate the application of the theoretical framework. We compute the index by inspection for several hybrid systems and showcase how Conley index theory can be applied via the correspondence with the suspension semiflow $\Phi_\cH$ to characterize the underlying hybrid dynamics. Throughout, we take homology coefficients over a field $\F$. Examples~\ref{subsec:BouncingBall}--~\ref{subsec:Neuron} satisfy the Trapping Guard Condition while Example~\ref{subsec:BrokenThermostat} does not.

\subsection{Bouncing Ball: Zeno Trajectories}
\label{subsec:BouncingBall}

We revisit the bouncing ball (Example~\ref{ex:BouncingBall}, Figure~\ref{fig:bouncing_ball_traj}), defined by the hybrid system $\cH = (X, \varphi, G, r)$, with a coefficient of restitution $c \in [0,1)$. It is well known that all solutions of this system converge to the origin in finite time. Let $S = \setof{(0,0)}$ be the set representing the origin. If $0 < c < 1$, trajectories approaching $S$ are Zeno (Definition~\ref{defn:HybridTrajectory}).

We analyze this invariant set using the suspension semiflow. Note that $S \subset G$ and $r(0,0)=(0,0)$. We examine the structure of the corresponding invariant set $\susp S$ in the suspension space $\susp X$. Following the construction in Section~\ref{sec:suspension}, $\susp S$ is derived from the relaxed system over $S$. This involves the cylinder segment $S \times [0,1]$ in $X'$, where the endpoint $((0,0),1)$ is identified with $((0,0),0)$ since $r(0,0)=(0,0)$. Since $S$ is a point, $S\times[0,1]$ is an interval, and identifying its endpoints gives a circle, $S^1$.

The suspension semiflow $\Phi_\cH$ flows continuously around this circle with period $T=1$. Thus, the invariant set $S = \setof{(0,0)}$ of $\cH$ corresponds to a periodic orbit in $\Phi_\cH$.

Although that is enough to characterize the homological index, we compute the Suspension Conley Index $\SCH_*(S)$ by identifying an isolating neighborhood. Consider the energy function $E(x,y) = \frac{1}{2}y^2 + gx$. Note that the energy is conserved along the flow, since $\dot{E} = y\dot{y} + g\dot{x} = -gy + gy = 0$, and strictly decreases at impacts since if $(0,y) \in G$ with $y < 0$, then $E(r(0,y)) = \frac{1}{2}c^2y^2 < \frac{1}{2}y^2 = E(0,y)$. Thus the set
\[
  N_\varepsilon = \setdef{ (x,y) \in X }{ E(x,y) \leq \varepsilon }
\]
is an isolating neighborhood for $S$. The pair $(N_\varepsilon, \emptyset)$ is an index pair. By Theorem~\ref{thm:SuspensionIndexPair}, $\SCH_*(S) = H_*(\susp N_\varepsilon)$, and $\susp N_\varepsilon \simeq S^1$.

The index is $\SCH_*(S) = H_*(S^1)$, that is,
\[
  \SCH_k(S) = H_k(S^1) =
  \begin{cases}
    \F & \text{if } k=0, 1 \\
    0 & \text{otherwise.}
  \end{cases}
\]

In the semiflow setting, this is the Conley index of an attracting periodic orbit. Note that the suspension admits a Poincaré section. Since $\dim \SCH_0(S) = \dim \SCH_1(S) = 1$, Theorem~\ref{thm:McCord}(ii) confirms the existence of a periodic orbit in $\susp S$.

\subsection{Thermostat: (Hybrid) Periodic Orbit}
\label{subsec:Thermostat}

Consider a standard thermostat model controlling temperature $z$, switching between ON ($q=1$) and OFF ($q=0$).
In our setting, consider the hybrid system defined as follows:
\begin{itemize}
  \item $X = [0,100] \times \setof{0,1}$, where the first coordinate represents temperature and the second coordinate represents the heater state (0 = OFF, 1 = ON),
  \item $\varphi$ is the local semiflow generated by the ODE $\dot{z} = -z + z_0 + z_{\delta} \cdot q$, where $z_0$ is the base temperature and $z_{\delta}$ is the temperature increment when the heater is on,
  \item $G = [0,z_{\min}] \times \{0\} \cup [z_{\max},100] \times \{1\}$, where $z_{\min}, z_{\max} \in [0,100]$ with $z_{\min} < z_{\max}$ are the minimum and maximum temperature thresholds,
  \item $r(z,q) = (z,1-q)$ is the reset map.
\end{itemize}

We shall consider the following parameters as in Example 1.9 of \cite{GoebelSanfeliceTeel2012}.
\begin{equation*}
  \begin{array}{rlrl}
    z_0      &= 60.0, & \quad z_{\min} &= 70.0, \\
    z_{\delta} &= 30.0, & \quad z_{\max} &= 80.0.
  \end{array}
\end{equation*}

Let $S = [z_{\min},z_{\max}]\times \setof{0,1}$. Note that $S$ is an isolated invariant set. Under the suspension, the continuous segments of the trajectory are connected by the flow through the suspension cylinders at the resets. The resulting set $\susp S$ is a continuous periodic orbit in $\Phi_\cH$, also homeomorphic to $S^1$ (Figure~\ref{fig:thermostat}). Using the index pair $(N,L) = (X, \emptyset)$, we find that the Suspension Conley Index is given by
\[
  \SCH_k(S) = H_k(S^1) =
  \begin{cases}
    \F & \text{if } k=0, 1 \\
    0 & \text{otherwise,}
  \end{cases}
\]
which corresponds to the index of a stable periodic orbit. Unlike the bouncing ball, the thermostat satisfies $r(G)\cap G = \emptyset$, since $r(G) = [0,z_{\min}]\times\{1\} \cup [z_{\max},100]\times\{0\}$ and $z_{\min} < z_{\max}$. Since $\dim \SCH_0(S) = \dim \SCH_1(S) = 1$ and $r(G) \cap G = \emptyset$ excludes periodic points of $r$, Corollary~\ref{cor:PeriodicOrbitI} guarantees $S$ contains a hybrid periodic orbit.
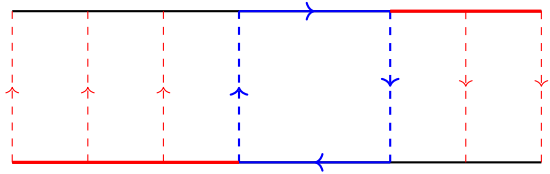
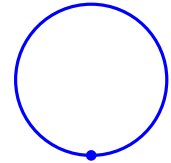
\begin{figure}[ht]
  \centering
  \begin{subfigure}{0.45\textwidth}
    \centering
    \adjustbox{max width=\linewidth}{
      \begin{tikzpicture}
        \draw[thick, black] (0,0) -- (7,0);
        \draw[thick, black] (0,2) -- (7,2);
        \draw[thick, blue, midarrow=0.5] (5,0) -- (3,0);
        \draw[thick, blue, midarrow=0.5] (3,2) -- (5,2);
        \draw[thick, blue, midarrow=0.5,dashed] (3,0) -- (3,2);
        \draw[thick, blue, midarrow=0.5,dashed] (5,2) -- (5,0);
        \draw[red, midarrow=0.5,dashed] (0,0) -- (0,2);
        \draw[red, midarrow=0.5,dashed] (1,0) -- (1,2);
        \draw[red, midarrow=0.5,dashed] (2,0) -- (2,2);
        \draw[red, midarrow=0.5,dashed] (6,2) -- (6,0);
        \draw[red, midarrow=0.5,dashed] (7,2) -- (7,0);
        \draw[very thick, red] (0,0) -- (3,0);
        \draw[very thick, red] (5,2) -- (7,2);
    \end{tikzpicture}}
    \caption{$\cH = (X,\varphi,G,r)$}
    \label{subfig:thermostat_hds}
  \end{subfigure}
  \hfill
  \begin{subfigure}{0.45\textwidth}
    \centering
    \begin{tikzpicture}
      \draw[very thick, blue] (0,0) circle (1.0);
      \fill[blue] (0,-1) circle (2pt);
    \end{tikzpicture}
    \caption{$\susp N \simeq S^1$}
    \label{subfig:thermostat_homology}
  \end{subfigure}
  \caption{Thermostat hybrid system on $[0,100]\times\setof{0,1}$, with the trajectory through $S$ in blue and the guard set $G$ in red. The index pair $(N,L) = (X,\emptyset)$ produces $\susp N \simeq S^1$ and $\susp L = \emptyset$, which has the index of an attracting periodic orbit.}
  \label{fig:thermostat}
\end{figure}

\subsection{Cubic Reset Map: Bistability and Morse Decompositions}
\label{subsec:CubicMap}
We consider the purely discrete system $X=[-1,1]$, $G=X$ and $r(x) = (3x - x^3)/2$. Since $G = X$, the impact time is always zero, and the Trapping Guard Condition holds with $U = X$ and $H(x,s) = x$.
The relaxed space $X'$ is given by $G \times [0,1] = [-1,1] \times [0,1]$, where $\varphi'(t,(x,s))=(x,s+t)$ for $0 \leq t \leq 1-s$, $G' = [-1,1] \times \setof{1}$, and $r'(x,1) = (r(x),0)$ (Figure~\ref{subfig:unstable}). Hence $\susp X$ is the mapping torus of $r$, obtained by identifying $(x,1)$ with $(r(x),0)$. The map $r$ has three fixed points, $x=-1, 0, 1$, which form the isolated invariant sets $S_2 = \setof{-1}$, $S_0 = \setof{0}$, and $S_1 = \setof{1}$ of $\cH$. Their suspensions are periodic orbits of $\Phi_\cH$ of period $1$ (Figure~\ref{fig:unstable}).
\begin{figure}[ht]
  \centering
  \begin{subfigure}{0.5\textwidth}
    \centering
    \begin{tikzpicture}[xscale=2.5, yscale=2]
      \draw[fill=gray!10, draw=none] (-1,0) rectangle (1,1);
      \draw[thick] (-1,0) -- (1,0);
      \draw[very thick, red] (-1,1) -- (1,1);
      \draw[thin] (-1,0) -- (-1,1);
      \draw[thin] (1,0) -- (1,1);

      \node[right] at (1,0) {$x$};
      \node[left] at (-1,1) {$s$};

      \draw[thick, violet, midarrow=0.3] (0,0) -- (0,1);
      \node[violet, below] at (0,0) {$S_0$};
      \draw[thick, blue, midarrow=0.3] (1,0) -- (1,1);
      \node[blue, below] at (1,0) {$S_1$};
      \draw[thick, blue, midarrow=0.3] (-1,0) -- (-1,1);
      \node[blue, below] at (-1,0) {$S_2$};

      \node[red, above] at (0.7,1) {$G'$};

      \draw[thick, gray, midarrow=0.3] (0.4,0) -- (0.4,1);
      \draw[thick, gray, dashed, midarrow=0.3] (0.4,1) -- (0.57,0);
      \draw[thick, gray, midarrow=0.3] (0.57,0) -- (0.57,1);
      \draw[thick, gray, midarrow=0.3] (-0.5,0) -- (-0.5,1);
      \draw[thick, gray, dashed, midarrow=0.3] (-0.5,1) -- (-0.69,0);
      \draw[thick, gray, midarrow=0.3] (-0.69,0) -- (-0.69,1);
    \end{tikzpicture}
    \caption{$\cH' = (X',\varphi',G',r')$}
    \label{subfig:unstable}
  \end{subfigure}
  \hfill
  \begin{subfigure}{0.45\textwidth}
    \centering
    \begin{tikzpicture}
      \draw[->] (-1.4, 0) -- (1.4, 0) node[below] {$x_{n}$};
      \draw[->] (0, -1.4) -- (0, 1.4) node[left] {$x_{n+1}$};
      \draw[domain=-1.25:1.25, samples=100, smooth, variable=\x, thick, blue] plot ({\x}, {(3*\x - \x*\x*\x)/2});
      \draw[dashed, gray] (-1.3, -1.3) -- (1.3, 1.3) node[pos=0.8, below right] {$x_{n+1}=x_n$};

      \fill[violet] (0,0) circle (2pt);
      \fill[blue] (1,1) circle (2pt);
      \fill[blue] (-1,-1) circle (2pt);
    \end{tikzpicture}
    \caption{Graph of $(3x-x^3)/2$ and its fixed points}
    \label{subfig:cubic_map}
  \end{subfigure}
  \caption{Visualization of the relaxed hybrid system of the cubic map.}
  \label{fig:unstable}
\end{figure}

For the attractors $S_{1}$ and $S_{2}$, the indices are computed as in the previous example:
\[
  \SCH_k(S_{1}) = \SCH_k(S_{2}) = H_k(S^1).
\]
For $S_{0}$, the corresponding invariant set $\susp S_{0}$ in the suspension semiflow $\Phi_\cH$ is an unstable periodic orbit where the instability comes from $\tfrac{dr}{dx}(0) = \tfrac{3}{2} > 1$. By Example~\ref{ex:conley_periodic}, the Conley index of a hyperbolic periodic orbit with unstable manifold of dimension $n+1$ is $CH_{k} = \F$ for $k=n,n+1$ and zero otherwise. Here the unstable manifold of $\susp S_{0}$ has dimension $2$, so $n=1$ and
\[
  \SCH_k(S_{0}) =
  \begin{cases}
    \F & \text{if } k=1, 2 \\
    0 & \text{otherwise.}
  \end{cases}
\]
Note that the system admits a Morse representation $\sM=\{S_{0}, S_{1}, S_{2}\}$ with the ordering $S_{1} < S_{0}$ and $S_{2} < S_{0}$, which captures the bistability. We remark that, in this case, the Suspension Conley Index of the map is equivalent to the reduced mapping torus index \cite{Weilandt2019} of the map $r$; see also \cite{Floer1990,Robbin1989}.

\subsection{Rimless Wheel: A Simple Planar Walking Model}
\label{subsec:RimlessWheel}

We discuss the rimless wheel model as in \cite{Rimless}. The model represents a simple planar walking model, where a wheel with spokes rolls down on an inclined plane \cite{Coleman1998}. Given the slope $\gamma$ and the spoke angle $2\alpha$, one considers the angle $\theta$ between the vertical and the spoke to define the hybrid dynamical system as follows:
\begin{itemize}
  \item $X = \setdef{ (\theta,\dot{\theta}) \in [\gamma-\alpha,\alpha+\gamma] \times \R }{ \frac{1}{2}\dot{\theta}^2+\cos(\theta) \leq E_{\max} }$ for some $E_{\max} > 0$,
  \item $\varphi$ is generated by the ODE $\ddot{\theta} = \sin(\theta)$,
  \item $G = \setdef{ (\theta,\dot{\theta}) \in X } { \theta = \alpha + \gamma , \dot{\theta}\geq 0 } \cup \setdef{ (\theta,\dot{\theta}) \in X } { \theta = \gamma - \alpha , \dot{\theta}\leq 0 }$,
  \item $r(\theta,\dot{\theta}) = (2\gamma - \theta, \cos(2\alpha) \dot{\theta})$.
\end{itemize}
Consider the parameters $\alpha=0.4$, $\gamma=0.2$ and $E_{\max}=3/2$, and notice that the continuous part is governed by the pendulum equation $\ddot{\theta}=\sin(\theta)$. Therefore, the quantity
\[
  E(\theta,\dot{\theta}) = \frac{1}{2}\dot{\theta}^2+\cos(\theta)
\]
is preserved along trajectories on $X\setminus G$. Indeed, note that
\[
  \frac{d}{dt} \left( \frac{1}{2}\dot{\theta}^2+\cos(\theta) \right)
  = \dot{\theta}\ddot{\theta}-\dot{\theta}\sin(\theta)
  = \dot{\theta}(\ddot{\theta}-\sin(\theta))
  = 0.
\]
Let $p = (-0.2,0.7)$ and $q = (-0.2,0.3)$. We show that the pair $(N_{p,q},\emptyset)$ is a hybrid index pair for $\cH$ (Figure~\ref{fig:rimless_wheel}), where
\[
  N_{p,q} = \setdef{ (\theta,\dot{\theta}) \in X }{ \dot{\theta} \geq 0,\ E(q) \leq E(\theta,\dot{\theta}) \leq E(p) }.
\]It is enough to verify what happens to $E$ after the reset since the solutions starting at $p$ and $q$ will bound our set. Solving for $p_f = \varphi(\stoptime(p),p)$ and $q_f = \varphi(\stoptime(q),q)$ using $E$ at $\theta=0.6$ yields the final points in the segment
\begin{align*}
  p_f & = (0.6,\sqrt{(0.7)^2+2(\cos(0.2)-\cos(0.6))}), \\
  q_f & = (0.6,\sqrt{(0.3)^2+2(\cos(0.2)-\cos(0.6))}).
\end{align*}
Thus,
\begin{align*}
  E(r(p_f)) & = \frac{1}{2}[\cos(0.8)]^2\left[ (0.7)^2+2(\cos(0.2)-\cos(0.6)) \right] + \cos(-0.2) \approx 1.174, \\
  E(r(q_f)) & = \frac{1}{2}[\cos(0.8)]^2\left[ (0.3)^2+2(\cos(0.2)-\cos(0.6)) \right] + \cos(-0.2) \approx 1.077.
\end{align*}
Since $E(q) \approx 1.025$ and $E(p) \approx 1.225$, we have $E(r(p_f)) < E(p)$ and $E(r(q_f)) > E(q)$. Thus reset images from both boundaries land strictly inside $N_{p,q}$, so $r(N_{p,q}\cap G) \subset \Int(N_{p,q})$. Let $A = \Inv(N,\cH)$ be the maximal invariant set in $N$ and observe that its index is $\SCH_*(A) = H_*(S^1)$.
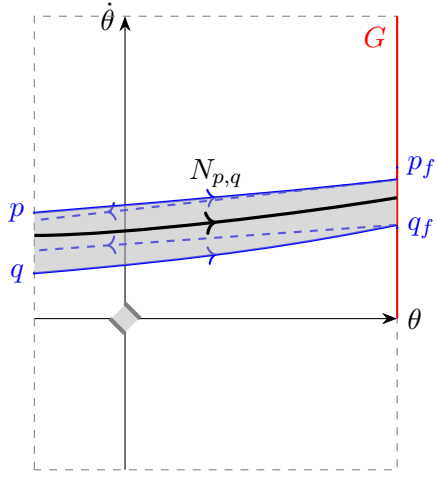
\begin{figure}
  \centering
  \begin{tikzpicture}[x=6cm,y=2cm]
    \def\alpha{0.4}
    \def\gamma{0.2}
    \def\a{0.1}

    \draw[-{Stealth[length=2mm]}] (\gamma-\alpha,0) -- (\alpha+\gamma,0) node[right] {$\theta$};
    \draw[-{Stealth[length=2mm]}] (0,-1.3) -- (0,1.4) node[left] {$\dot{\theta}$};

    \draw[dashed, gray] plot[domain=-0.2:0.6, samples=40] (\x, {sqrt(2*(1.5-cos(\x r)))}) -- plot[domain=0.6:-0.2, samples=40] (\x, {-sqrt(2*(1.5-cos(\x r)))}) -- cycle;
    \draw[red, thick] (\alpha+\gamma, 0) -- (\alpha+\gamma, {sqrt(2*(1.5-cos(0.6 r)))}) node[below left] {$G$};
    \draw[red, thick] (\gamma-\alpha, 0) -- (\gamma-\alpha, {-sqrt(2*(1.5-cos(0.2 r)))}) node[above right] {$G$};

    \fill[blue] (-0.2, 0.7) circle (0.5pt) node[left] {$p$};
    \fill[blue] (-0.2, 0.3) circle (0.5pt) node[left] {$q$};
    \draw[blue, thick, midarrow=0.5] (-0.2, 0.7) .. controls (0.2, 0.8) and (0.4, 0.85) .. (0.6, 0.92);
    \draw[blue, thick, midarrow=0.5] (-0.2, 0.3) .. controls (0.2, 0.4) and (0.4, 0.5) .. (0.6, 0.62);

    \fill[blue] (0.6, 1.0) circle (0.5pt) node[right] {$p_f$};
    \fill[blue] (0.6, 0.6) circle (0.5pt) node[right] {$q_f$};

    \draw[blue, thick, dashed, midarrow=0.8] (0.6, 0.92) -- (-0.2, 0.65);
    \draw[blue, thick, dashed, midarrow=0.8] (0.6, 0.62) -- (-0.2, 0.45);

    \fill[gray!60, opacity=0.5] (-0.2, 0.7) .. controls (0.2, 0.8) and (0.4, 0.85) .. (0.6, 0.92) -- (0.6, 0.62) .. controls (0.4, 0.5) and (0.2, 0.4) .. (-0.2, 0.3) -- cycle;
    \node at (0.2, 0.8) [above] {$N_{p,q}$};

    \draw[black, very thick, midarrow=0.5] (-0.2, 0.55) .. controls (0, 0.55) and (0.4, 0.7) .. (0.6, 0.8);

    \node[fill=gray!30, rectangle, minimum width={2*\a*1cm},minimum height={2*\a*1cm}, rotate=45, line width=0.5pt] at (0,0) {};
    \draw[black!50, ultra thick] (0,2*\a*1cm) -- (2*\a*1cm,0);
    \draw[black!50, ultra thick] (0,-2*\a*1cm) -- (-2*\a*1cm,0);
  \end{tikzpicture}
  \caption{Illustration of index pairs for a limit cycle and an equilibrium in the Rimless Wheel.}
  \label{fig:rimless_wheel}
\end{figure}
We also consider the equilibrium point $R = \setof{(0,0)}$. Since $R \cap G = \emptyset$ (as $\gamma - \alpha < 0 < \alpha + \gamma$), the hybrid dynamics around $R$ are purely continuous. The suspension $\susp R$ is simply $\pi(\iota(R))\cong R$ itself, a fixed point for $\Phi_\cH$. The Suspension Conley Index is therefore the classical Conley index of $R$ under $\varphi$. Note that $R$ consists of a saddle point with a one-dimensional unstable manifold. The Conley index of hyperbolic fixed points is well known, and in this case is given by
\[
  \SCH_k(R) =
  \begin{cases}
    \F & \text{if } k=1 \\
    0 & \text{otherwise.}
  \end{cases}
\]
Theorem~\ref{thm:McCord}(i) applies with $n=1$, so the index alone shows that $R$ contains a fixed point. This illustrates how the framework handles invariant sets arising from both the hybrid interaction and the underlying continuous flow.

\subsection{Neuron model: Quadratic Integrate-and-Fire}
\label{subsec:Neuron}

We analyze a simple spiking neuron model based on \cite{Izhikevich2003}. The 2-dimensional model has variables $v$, for the membrane potential, and $u$, for the recovery variable. The hybrid dynamical system can be written as $\cH = (X,\varphi,G,r)$ where
\begin{itemize}
  \item $X=N$, where $N$ is the compact region constructed below, contained in $\setdef{(v,u)\in\R^2}{v\leq v_{peak}}$, with $v_{peak}>0$.
  \item $\varphi$ is generated by the ODE
    \[
      \begin{cases}
        C\dot{v} & = k(v-v_r)(v-v_t) - u + I \\
        \dot{u} & = a(b(v-v_r)-u)
      \end{cases}
    \]
    where $I$ is the input current and $C, k, a > 0$, $b < 0$, $v_r < v_t$ are parameters.
  \item $G = \setdef{(v,u) \in X}{v = v_{peak}}$,
  \item $r(v,u) = (c,u+d)$ where $c < v_{peak}$ is the reset voltage and $d>0$ is a recovery increment.
\end{itemize}

We consider the parameters $C = 100$, $k = 0.7$, $v_r = -60$, $v_t = -40$, $a = 0.03$, $b = -2$, $I = 70$, $c = -50$, $d = 100$, and $v_{peak} = 35$. The quadratic term pushes $v$ towards the threshold, producing a spike. At the peak $v_{peak}$, the voltage is reset to $c$ and the recovery variable is increased by $d$. The recovery variable acts as negative feedback since it increases at the spike and suppresses activity until it decays.

On the guard $G$, we have that $\dot{v} = (k(v_{peak}-v_r)(v_{peak}-v_t) - u + I)/C$. For the parameters of interest, $u\leq u_1<k(v_{peak}-v_r)(v_{peak}-v_t)+I$, so $\dot{v}>0$ on $G$. Therefore, the flow is transverse to $G$ and there is a neighborhood of $G$ in $X$ that reaches it in finite time.
\begin{figure}[ht]
  \centering
  \begin{tikzpicture}[xscale=3.5, yscale=1.8]
    \fill[gray!30, opacity=0.6]
    (-1,-1) -- (0.5,-1) -- (0.5,0.53) -- (-0.48,0.53) -- (-0.48,2) -- (-1,2) -- cycle;
    \draw[thick]
    (-1,-1) -- (0.5,-1) -- (0.5,0.53) -- (-0.48,0.53) -- (-0.48,2) -- (-1,2) -- cycle;
    \node at (-0.85, 1.7) {$N$};

    \draw[red, very thick] (0.5, -1) -- (0.5, 0.53);
    \node[red, right] at (0.5, -0.2) {$G \cap N$};

    \draw[red, thick, dashed] (-0.61, -0.67) -- (-0.61, 0.87);
    \node[red, below] at (-0.61, -0.67) {\small $r(N\cap G)$};

    \draw[black!50, thick] (-1, 2.0) parabola bend (-0.61, 0) (-0.22, 2.0);
    \node[black!50] at (-0.05, 2.00) {\small $\dot{v}=0$};

    \draw[black!50, thick] (-1, 0.13) -- (0.5, -0.63);
    \node[black!50] at (0.65, -0.75) {\small $\dot{u}=0$};

    \draw[black, very thick, midarrow=0.5]
    (-0.61, 0.21) .. controls (-0.70, 0.15) and (-0.70, 0.02) .. (-0.65, -0.02)
    .. controls (-0.60, -0.08) and (-0.45, -0.10) .. (-0.20, -0.11)
    .. controls (0.05, -0.12) and (0.30, -0.12) .. (0.50, -0.12);
    \draw[black, very thick, dashed, midarrow=0.8] (0.5, -0.12) -- (-0.61, 0.21);

    \fill[black] (-0.61, 0.21) circle (1pt);
    \fill[black] (0.5, -0.12) circle (1pt);

    \node[below, font=\scriptsize] at (-1, -1) {$v_0$};
    \node[below, font=\scriptsize] at (-0.61, -1) {$c$};
    \node[below, font=\scriptsize] at (-0.5, -1) {$v_1$};
    \node[below, font=\scriptsize] at (0.5, -1) {$v_{peak}$};

    \node[left, font=\scriptsize] at (-1, -1) {$u_0$};
    \node[left, font=\scriptsize] at (-1, 0.53) {$u_1$};
    \node[left, font=\scriptsize] at (-1, 2) {$u_2$};

  \end{tikzpicture}
  \caption{L-shaped forward-invariant region $N$ for the simple neuron model. The limit cycle (solid black) with reset jump (dashed black) is contained in $N$. The nullclines are shown in gray. The constant values are $v_0=-80$, $v_1=-40$, $u_0=-300$, $u_1=160$, $u_2=600$.}

  \label{fig:neuron}
\end{figure}
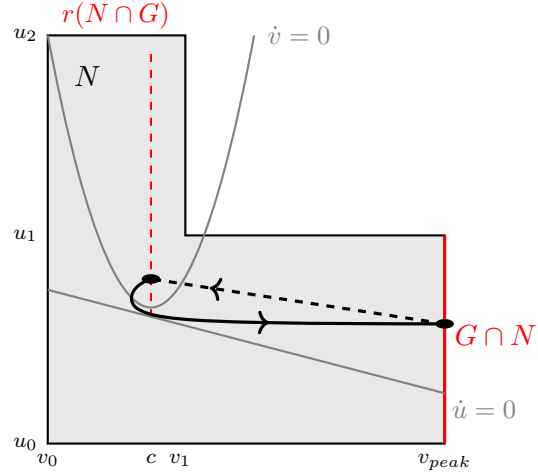

We now construct a compact set $N$ that is forward-invariant. The key insight is constructing a bounded neighborhood whose flow is transverse at each point. We do this by evaluating the vector field along vertical or horizontal lines with the knowledge of the nullclines. The state space $X=N$ is the L-shaped compact set defined by
\[
  N = ([v_0, v_1] \times [u_0, u_2]) \cup ([v_1, v_{peak}] \times [u_0, u_1]).
\]
Along the boundary of $N$ outside $G$, the vector field points into $N$. Together with the fact that $r(N \cap G) \subseteq \Int(N)$, this shows that the pair $(N,\emptyset)$ is an index pair for the maximal invariant set contained in $N$. Let $A = \Inv(N, \cH)$ and observe that its index is $\SCH_*(A) = H_*(S^1)$. Therefore, by Corollary~\ref{cor:PeriodicOrbitI}, $A$ contains a periodic orbit of $\cH$. We highlight the simulated limit cycle in Figure~\ref{fig:neuron}.

\subsection{``Broken'' Thermostat: A Cautionary Example}
\label{subsec:BrokenThermostat}

We remark that this example violates the Trapping Guard Condition and falls outside the formal theory developed in previous sections. It serves to illustrate that the TGC is essential for the index theory to behave as expected. We compute the index directly from the homological information of the pairs, without the theoretical assurances provided by the hybrid suspension semiflow.

Consider a variation of the thermostat (Example~\ref{subsec:Thermostat}) where the sensor is not working properly. In particular, we keep the same parameters but modify the guard to be simply
\[
  G = \setof{ (z_{\min},0), (z_{\max},1)}.
\]
Let $z_1 = z_0+z_\delta$.
If the system starts in the ON state ($q=1$) with $z > z_{\max}$, the temperature goes to $z_1$. If the system starts in the OFF state ($q=0$) with $z < z_{\min}$, the temperature goes to $z_0$. In either case, the initial condition does not go through the guard set at any point, so the system either goes to $(z_0,0)$ or $(z_1,1)$.

Let $S = [z_{\min},z_{\max}]\times\setof{0,1}$. Note that $S$ is an isolated invariant set that corresponds to the desired oscillatory behavior with the temperature decreasing when the thermostat is OFF ($q=0$) until it reaches $z_{\min}$, then ON ($q=1$) until it reaches $z_{\max}$, and repeating. Given that slight perturbations at $(z_{\min},0)$ and $(z_{\max},1)$ lead the system to the attracting points $(z_0,0)$ and $(z_1,1)$, one would expect $S$ to exhibit the index of an unstable periodic orbit.

Although the system does not satisfy the TGC and thus falls outside the formal theory developed in Section~\ref{sec:suspension}, we can still construct a suspension-like space by attaching cylinders to the guard and taking the quotient:
\[
  \susp X \cong \frac{(G \times [0,1])\bigsqcup X } {(x,0)\sim x, (x,1)\sim r(x)}.
\]
We emphasize that this is a heuristic construction for illustrative purposes, not a rigorous application of Definition~\ref{defn:SuspensionSemiflow} because the suspension semiflow $\Phi_\cH$ does not extend continuously across the guard. Consider the pair of compact sets
\begin{align*}
  N & = X, \\
  L & = \left([0,z_{\min}-\varepsilon] \times \setof{0}\right) \cup \left( [z_{\max}+\varepsilon,100]\times\setof{1}\right),
\end{align*}
for sufficiently small $\varepsilon>0$. Then $(N,L)$ satisfies the conditions for a hybrid index pair, that is, $\Inv(N\setminus L, \cH) = S$, $L$ is positively invariant, and $L$ is the exit set for $N$. The pair, the suspension $\susp X$, and the homotopy type of the quotient are illustrated in Figure~\ref{fig:broken_thermostat}.
\begin{figure}[ht]
  \centering
  \begin{subfigure}{0.32\textwidth}
    \centering
    \adjustbox{max width=\linewidth}{
      \begin{tikzpicture}[scale=0.7]
        \useasboundingbox (-0.2,-0.3) rectangle (7.2,2.3);
        \draw[thick, black] (0,0) -- (7,0);
        \draw[thick, black] (0,2) -- (7,2);
        \draw[thick, red, midarrow=0.5] (5,0) -- (3,0);
        \draw[thick, red, midarrow=0.5] (3,2) -- (5,2);
        \draw[thick, red, midarrow=0.5,dashed] (3,0) -- (3,2);
        \draw[thick, red, midarrow=0.5,dashed] (5,2) -- (5,0);
        \node[circle, fill, red, inner sep=2pt] at (3,0) {};
        \node[circle, fill, red, inner sep=2pt] at (5,2) {};
        \node[circle, fill, blue, inner sep=2pt] at (1.5,0) {};
        \node[circle, fill, blue, inner sep=2pt] at (6,2) {};
        \draw[blue,midarrow=0.5] (2.7,0) -- (1.5,0);
        \draw[blue,midarrow=0.5] (0,0) -- (1.5,0);
        \draw[blue,midarrow=0.5] (5.3,2) -- (6,2);
        \draw[blue,midarrow=0.5] (7,2) -- (6,2);
    \end{tikzpicture}}
    \caption{$\cH = (X,\varphi,G,r)$}
    \label{subfig:broken_hds}
  \end{subfigure}
  \hfill
  \begin{subfigure}{0.32\textwidth}
    \centering
    \adjustbox{max width=\linewidth}{
      \begin{tikzpicture}[scale=0.7]
        \useasboundingbox (-0.2,-0.3) rectangle (7.2,2.3);
        \draw[ultra thick, blue] (0,0) -- (3,0);
        \draw[ultra thick, blue] (5,2) -- (7,2);
        \draw[thick, black] (3,0) -- (7,0);
        \draw[thick, black] (0,2) -- (5,2);
        \draw[thick, black] (3,0) -- (3,2);
        \draw[thick, black] (5,0) -- (5,2);
    \end{tikzpicture}}
    \caption{$(\susp N, \susp L)$}
    \label{subfig:broken_suspension}
  \end{subfigure}
  \hfill
  \begin{subfigure}{0.32\textwidth}
    \centering
    \begin{tikzpicture}[scale=0.7]
      \useasboundingbox (-2.2,-1.3) rectangle (2.2,1.3);
      \draw[thick, black] (-1,0) circle (1.0);
      \draw[thick, black] (1,0) circle (1.0);
      \fill[blue] (0,0) circle (2.5pt);
      \node[blue, right] at (0.05,-0.05) {\small $[\susp L]$};
    \end{tikzpicture}
    \caption{$\susp N/\susp L \simeq S^1 \vee S^1$}
    \label{subfig:broken_quotient}
  \end{subfigure}
  \caption{The broken thermostat (a) as a hybrid dynamical system with isolated invariant sets highlighted, (b) the suspension space $\susp X=\susp N$ with $\susp L$ highlighted in blue, and (c) the homotopy type of the quotient $\susp N/\susp L$.}
  \label{fig:broken_thermostat}
\end{figure}
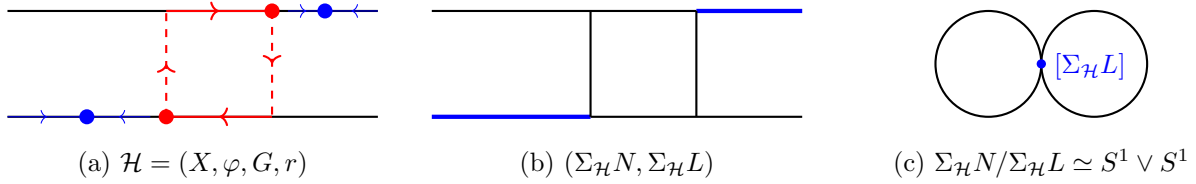

The homological Conley index is computed in the suspension. Set $\susp N = \susp X$ and let $\susp L \subset \susp X$ be the image of $L$ under the natural inclusion $X \hookrightarrow \susp X$. Since $L \cap G = \emptyset$, $\susp L$ is identified with $L$ as two disjoint closed intervals (Figure~\ref{subfig:broken_suspension}). Collapsing $\susp L$ to a base point $[\susp L]$, the quotient $\susp N/\susp L$ is homotopy equivalent to $S^1 \vee S^1$ (Figure~\ref{subfig:broken_quotient}). Thus
\[
  H_k(\susp N, \susp L) \cong \tilde H_k(\susp N/\susp L) =
  \begin{cases}
    \F \oplus \F & \text{if } k = 1, \\
    0 & \text{otherwise.}
  \end{cases}
\]
Note that $X/L$ alone is a wedge of two intervals, hence contractible, while the figure-eight is produced entirely by the cylinder identifications. This shows that the dictionary of Conley indices does not transfer to the hybrid setting without the suspension.

\section{Conclusion}
\label{sec:conclusion}

This work bridges the gap between topological dynamics and hybrid dynamical systems, providing a framework for analyzing hybrid dynamics through Conley Index Theory. By leveraging existing frameworks for dynamical systems and algebraic topology in a continuous setting, we analyzed the topological structure of hybrid systems that satisfy the trapping guard condition \cite{Kvalheim2021}. The results show the potential of this approach in reconstructing the dynamics of hybrid systems and identifying their invariant sets.

Future work will focus on extending these methods to a broader class of hybrid systems and exploring the implications of the results in practical applications. Towards exploratory goals, we aim to develop efficient computational tools for analyzing hybrid systems by improving on the existing Python libraries of Conley-Morse-Graph Database (CMGDB) \cite{Zin2009,CMGDB} and Computational Homology Project (pyCHomP) \cite{pyCHomP,ComputationalHomologyTextbook} to compute the Conley index. These libraries build on algorithmic foundations for chain recurrence, Conley index computation, connection matrices, and lattice and combinatorial representations of global dynamics \cite{BanKalies2006,BoczkoKaliesMischaikow2007,Bush2012,RookFields,HarkerMischaikowSpendlove2021,kaliesvandervorst2024priestley,KaliesKastiVandervorst2018,KaliesMischaikowVandervorst2005,Mrozek1999}. We highlight existing computational tools, such as MORALS \cite{morals2024}, that analyze high-dimensional continuous controllers in robotics using Conley theory, providing insight into the underlying structure of the dynamics.

While the results of this manuscript rely on the correspondence with a continuous semiflow supplied by the trapping guard condition, combinatorial descriptions of recurrence~\cite{KaliesMischaikowVandervorst2005}, attractors~\cite{KaliesMischaikowVandervorst2014,KaliesMischaikowVandervorst2016}, and the Conley index~\cite{KaliesMischaikowVandervorst2022,HarkerKokubuMischaikowPilarczyk2016} provide formulations of the topological structure of dynamics that need not assume continuity. It remains open to what extent algorithmic approaches can accommodate discontinuous dynamical systems.

\section*{Acknowledgments}
The authors thank Ram Vasudevan for early discussions on this work and for the suggested applications. This work was supported by the Air Force Office of Scientific Research under award numbers FA9550-23-1-0400 and FA9550-23-1-0011.

\bibliographystyle{elsarticle-num}
\bibliography{references}

@article{BanKalies2006,
  author        = {Ban, Hyunju and Kalies, William D.},
  title         = {A computational approach to {C}onley's decomposition theorem},
  journal       = {J. Comput. Nonlinear Dynam.},
  volume        = {1},
  number        = {4},
  pages         = {312--319},
  year          = {2006},
  month         = oct,
  publisher     = {ASME International},
  doi           = {10.1115/1.2338651}
}

@article{BoczkoKaliesMischaikow2007,
  author        = {Boczko, Erik and Kalies, William D. and Mischaikow, Konstantin},
  title         = {Polygonal approximation of flows},
  journal       = {Topology Appl.},
  volume        = {154},
  number        = {13},
  pages         = {2501--2520},
  year          = {2007},
  doi           = {10.1016/j.topol.2006.04.033}
}

@article{Mrozek1999,
  author        = {Mrozek, Marian},
  title         = {An Algorithmic Approach to the {C}onley Index Theory},
  journal       = {J. Dynam. Differential Equations},
  volume        = {11},
  number        = {4},
  pages         = {711--734},
  year          = {1999},
  publisher     = {Springer Nature},
  doi           = {10.1023/A:1022615629693}
}

@article{Bush2012,
  author        = {Bush, Justin and Gameiro, Marcio and Harker, Shaun and Kokubu, Hiroshi and Mischaikow, Konstantin and Obayashi, Ippei and Pilarczyk, Pawe{\l}},
  title         = {Combinatorial-topological framework for the analysis of global dynamics},
  journal       = {Chaos},
  volume        = {22},
  number        = {4},
  pages         = {047508},
  year          = {2012},
  month         = dec,
  doi           = {10.1063/1.4767672},
  pmid          = {23278094}
}

@book{Conley1978,
  author        = {Conley, Charles C.},
  title         = {Isolated Invariant Sets and the Morse Index},
  series        = {CBMS Regional Conference Series in Mathematics},
  volume        = {38},
  year          = {1978},
  publisher     = {American Mathematical Society},
  address       = {Providence, R.I.}
}

@book{GoebelSanfeliceTeel2012,
  author        = {Goebel, Rafal and Sanfelice, Ricardo G. and Teel, Andrew R.},
  title         = {Hybrid dynamical systems},
  year          = {2012},
  month         = feb,
  publisher     = {Princeton University Press},
  address       = {Princeton, NJ},
  language      = {en}
}

@article{HarkerKokubuMischaikowPilarczyk2016,
  author        = {Harker, Shaun and Kokubu, Hiroshi and Mischaikow, Konstantin and Pilarczyk, Pawe{\l}},
  title         = {Inducing a map on homology from a correspondence},
  journal       = {Proc. Amer. Math. Soc.},
  volume        = {144},
  number        = {4},
  pages         = {1787--1801},
  year          = {2016},
  doi           = {10.1090/proc/12812}
}

@article{HarkerMischaikowSpendlove2021,
  author        = {Harker, Shaun and Mischaikow, Konstantin and Spendlove, Kelly},
  title         = {A computational framework for connection matrix theory},
  journal       = {J. Appl. Comput. Topol.},
  volume        = {5},
  number        = {3},
  pages         = {459--529},
  year          = {2021},
  month         = sep,
  publisher     = {Springer Science and Business Media LLC},
  doi           = {10.1007/s41468-021-00073-3}
}

@article{KaliesKastiVandervorst2018,
  author        = {Kalies, William D. and Kasti, Dinesh and Vandervorst, Robert C. A. M.},
  title         = {An Algorithmic Approach to Lattices and Order in Dynamics},
  journal       = {SIAM J. Appl. Dyn. Syst.},
  volume        = {17},
  number        = {2},
  pages         = {1617--1649},
  year          = {2018},
  doi           = {10.1137/17M1139606}
}

@article{KaliesMischaikowVandervorst2005,
  author        = {Kalies, William D. and Mischaikow, Konstantin and Vandervorst, Robert C. A. M.},
  title         = {An algorithmic approach to chain recurrence},
  journal       = {Found. Comput. Math.},
  volume        = {5},
  number        = {4},
  pages         = {409--449},
  year          = {2005},
  month         = nov,
  publisher     = {Springer Science and Business Media LLC},
  doi           = {10.1007/s10208-004-0163-9}
}

@article{KaliesMischaikowVandervorst2014,
  author        = {Kalies, William D. and Mischaikow, Konstantin and Vandervorst, Robert C. A. M.},
  title         = {Lattice structures for attractors {I}},
  journal       = {J. Comput. Dyn.},
  volume        = {1},
  number        = {2},
  pages         = {307--338},
  year          = {2014},
  publisher     = {American Institute of Mathematical Sciences (AIMS)},
  doi           = {10.3934/jcd.2014.1.307}
}

@article{KaliesMischaikowVandervorst2016,
  author        = {Kalies, William D. and Mischaikow, Konstantin and Vandervorst, Robert C. A. M.},
  title         = {Lattice Structures for Attractors {II}},
  journal       = {Found. Comput. Math.},
  volume        = {16},
  number        = {5},
  pages         = {1151--1191},
  year          = {2016},
  month         = oct,
  doi           = {10.1007/s10208-015-9272-x},
  issn          = {1615-3383}
}

@article{KaliesMischaikowVandervorst2022,
  author        = {Kalies, William D. and Mischaikow, Konstantin and Vandervorst, Robert C. A. M.},
  title         = {Lattice structures for attractors {III}},
  journal       = {J. Dynam. Differential Equations},
  volume        = {34},
  number        = {3},
  pages         = {1729--1768},
  year          = {2022},
  month         = sep,
  publisher     = {Springer Science and Business Media LLC},
  doi           = {10.1007/s10884-021-10056-8}
}

@inproceedings{morals2024,
  author        = {Vieira, Ewerton R. and Sivaramakrishnan, Aravind and Tangirala, Sumanth and Granados, Edgar and Mischaikow, Konstantin and Bekris, Kostas E.},
  title         = {{MORALS}: Analysis of High-Dimensional Robot Controllers via Topological Tools in a Latent Space},
  booktitle     = {IEEE International Conference on Robotics and Automation (ICRA)},
  pages         = {27--33},
  year          = {2024},
  publisher     = {IEEE},
  doi           = {10.1109/ICRA57147.2024.10610383}
}

@article{Kvalheim2021,
  author        = {Kvalheim, Matthew D. and Gustafson, Paul and Koditschek, Daniel E.},
  title         = {Conley's fundamental theorem for a class of hybrid systems},
  journal       = {SIAM J. Appl. Dyn. Syst.},
  volume        = {20},
  number        = {2},
  pages         = {784--825},
  year          = {2021},
  month         = jan,
  publisher     = {Society for Industrial \& Applied Mathematics (SIAM)},
  doi           = {10.1137/20M1336576}
}

@article{Goebel2023,
  author        = {Goebel, Rafal},
  title         = {A Direct Proof of {C}onley's Decomposition for Well-Posed Hybrid Inclusions},
  journal       = {Systems Control Lett.},
  volume        = {180},
  pages         = {105604},
  year          = {2023},
  doi           = {10.1016/j.sysconle.2023.105604},
  issn          = {0167-6911}
}

@book{HatcherAT,
  author        = {Hatcher, Allen},
  title         = {Algebraic Topology},
  year          = {2002},
  publisher     = {Cambridge University Press},
  address       = {Cambridge},
  isbn          = {0-521-79160-X}
}

@book{diBernardo2008Piecewise,
  author        = {di Bernardo, Mario and Budd, Christopher J. and Champneys, Alan R. and Kowalczyk, Piotr},
  title         = {Piecewise-smooth Dynamical Systems: Theory and Applications},
  series        = {Applied Mathematical Sciences},
  volume        = {163},
  pages         = {xxii+482},
  year          = {2008},
  publisher     = {Springer London},
  doi           = {10.1007/978-1-84628-708-4},
  isbn          = {978-1-84628-039-9}
}

@book{Schaft2000Hybrid,
  author        = {van der Schaft, Arjan and Schumacher, Hans},
  title         = {An Introduction to Hybrid Dynamical Systems},
  series        = {Lecture Notes in Control and Information Sciences},
  volume        = {251},
  pages         = {xiv+174},
  year          = {2000},
  publisher     = {Springer London},
  doi           = {10.1007/BFb0109998},
  isbn          = {978-1-4471-3916-4}
}

@article{kaliesvandervorst2024priestley,
  author        = {Kalies, William D. and Vandervorst, Robert C. A. M.},
  title         = {Priestley Duality and Representations of Global Dynamics},
  journal       = {Qual. Theory Dyn. Syst.},
  volume        = {24},
  number        = {4},
  pages         = {186},
  year          = {2025},
  doi           = {10.1007/s12346-025-01343-6}
}

@article{McCord1995,
  author        = {McCord, Christopher and Mischaikow, Konstantin and Mrozek, Marian},
  title         = {Zeta functions, periodic trajectories, and the {C}onley index},
  journal       = {J. Differential Equations},
  volume        = {121},
  number        = {2},
  pages         = {258--292},
  year          = {1995},
  month         = sep,
  publisher     = {Elsevier BV},
  doi           = {10.1006/jdeq.1995.1129}
}

@article{Mrozek1988,
  author        = {Mrozek, Marian},
  title         = {Periodic and stationary trajectories of flows and ordinary differential equations},
  journal       = {Univ. Iagel. Acta Math.},
  volume        = {27},
  pages         = {29--37},
  year          = {1988}
}

@article{McCord1988,
  author        = {McCord, Christopher},
  title         = {Mappings and homological properties in the {C}onley index theory},
  journal       = {Ergodic Theory Dynam. Systems},
  volume        = {8},
  pages         = {175--198},
  year          = {1988},
  month         = dec,
  publisher     = {Cambridge University Press (CUP)},
  doi           = {10.1017/S014338570000941X}
}

@inproceedings{Rimless,
  author        = {Shia, Victor and Vasudevan, Ram and Bajcsy, Ruzena and Tedrake, Russ},
  title         = {Convex computation of the reachable set for controlled polynomial hybrid systems},
  booktitle     = {53rd IEEE Conference on Decision and Control},
  volume        = {},
  number        = {},
  pages         = {1499--1506},
  year          = {2014},
  doi           = {10.1109/CDC.2014.7039612}
}

@phdthesis{Coleman1998,
  author        = {Coleman, Michael Jon},
  title         = {A stability study of a three-dimensional passive-dynamic model of human gait},
  year          = {1998},
  school        = {Cornell University}
}

@incollection{Conley1975,
  author        = {Conley, Charles C. and Smoller, Joel A.},
  title         = {The existence of heteroclinic orbits, and applications},
  booktitle     = {Dynamical Systems, Theory and Applications},
  series        = {Lecture Notes in Phys.},
  volume        = {38},
  pages         = {511--524},
  year          = {1975},
  publisher     = {Springer}
}

@article{Day2019,
  author        = {Day, Sarah and Frongillo, Rafael},
  title         = {Sofic shifts via {C}onley index theory: Computing lower bounds on recurrent dynamics for maps},
  journal       = {SIAM J. Appl. Dyn. Syst.},
  volume        = {18},
  number        = {3},
  pages         = {1610--1642},
  year          = {2019},
  month         = jan,
  publisher     = {Society for Industrial \& Applied Mathematics (SIAM)},
  doi           = {10.1137/18M1192007}
}

@incollection{Handbook2002,
  author        = {Mischaikow, Konstantin and Mrozek, Marian},
  title         = {Conley Index},
  booktitle     = {Handbook of Dynamical Systems},
  editor        = {Fiedler, Bernold},
  volume        = {2},
  pages         = {393--460},
  year          = {2002},
  publisher     = {Elsevier},
  doi           = {10.1016/S1874-575X(02)80030-3}
}

@article{MischaikowMrozek1995,
  author        = {Mischaikow, Konstantin and Mrozek, Marian},
  title         = {Chaos in the {L}orenz equations: A computer-assisted proof},
  journal       = {Bull. Amer. Math. Soc.},
  volume        = {32},
  number        = {1},
  pages         = {66--72},
  year          = {1995},
  doi           = {10.1090/S0273-0979-1995-00558-6}
}

@misc{RookFields,
  author        = {Gameiro, Marcio and Gedeon, Tom{\'a}{\v{s}} and Kokubu, Hiroshi and Mischaikow, Konstantin and Oka, Hiroe and Rivas, Bernardo and Vieira, Ewerton R. and Gameiro, Daniel},
  title         = {Global Dynamics of Ordinary Differential Equations: Wall Labelings, {C}onley Complexes, and Ramp Systems},
  year          = {2024},
  doi           = {10.48550/arXiv.2412.11078},
  eprint        = {2412.11078},
  archiveprefix = {arXiv},
  primaryclass  = {math.DS}
}

@misc{KaliesRivasWehbe2026,
  author        = {Kalies, William D. and Rivas, Bernardo and Wehbe, Tony},
  title         = {Computation of Attractor Lattices for a Class of Hybrid Systems},
  year          = {2026},
  doi           = {10.48550/arXiv.2610.15123},
  eprint        = {2610.15123},
  archiveprefix = {arXiv},
  primaryclass  = {math.DS}
}

@article{Alex2023,
  author        = {Dowling, K. Alex and Kalies, William D. and Vandervorst, Robert C. A. M.},
  title         = {Continuation sheaves in dynamics: Sheaf cohomology and bifurcation},
  journal       = {J. Differential Equations},
  volume        = {367},
  pages         = {124--198},
  year          = {2023},
  month         = sep,
  publisher     = {Elsevier BV},
  doi           = {10.1016/j.jde.2023.04.041}
}

@article{Zin2009,
  author        = {Arai, Zin and Kalies, William D. and Kokubu, Hiroshi and Mischaikow, Konstantin and Oka, Hiroe and Pilarczyk, Pawe{\l}},
  title         = {A Database Schema for the Analysis of Global Dynamics of Multiparameter Systems},
  journal       = {SIAM J. Appl. Dyn. Syst.},
  volume        = {8},
  number        = {3},
  pages         = {757--789},
  year          = {2009},
  doi           = {10.1137/080734935}
}

@misc{CMGDB,
  author        = {Gameiro, Marcio and Harker, Shaun},
  title         = {{CMGDB}: {Conley--Morse} Graph Database},
  year          = {2025},
  howpublished  = {\url{https://github.com/marciogameiro/CMGDB}},
  note          = {Version 1.3.2 [software]}
}

@misc{pyCHomP,
  author        = {Gameiro, Marcio and Harker, Shaun},
  title         = {{pyCHomP2}: {Python} Bindings for {CHomP} ({Computational Homology Project})},
  year          = {2025},
  howpublished  = {\url{https://github.com/marciogameiro/pyCHomP2}},
  note          = {Version 1.1.2 [software]}
}

@article{Weilandt2019,
  author        = {Weilandt, Frank},
  title         = {The {C}onley index for discrete dynamical systems and the mapping torus},
  journal       = {J. Appl. Comput. Topol.},
  volume        = {3},
  number        = {1-2},
  pages         = {119--138},
  year          = {2019},
  month         = jun,
  publisher     = {Springer Science and Business Media LLC},
  doi           = {10.1007/s41468-019-00027-w}
}

@article{Kaul1990,
  author        = {Kaul, Saroop K.},
  title         = {On impulsive semidynamical systems},
  journal       = {J. Math. Anal. Appl.},
  volume        = {150},
  number        = {1},
  pages         = {120--128},
  year          = {1990}
}

@article{Bonotto2024,
  author        = {Bonotto, Everaldo M. and Demuner, Daniela Paula and Souto, Ginnara Mexia},
  title         = {Recursiveness on impulsive dynamical systems: Minimality, non-wandering points, the center of {B}irkhoff and attractors},
  journal       = {J. Differential Equations},
  volume        = {410},
  pages         = {46--75},
  year          = {2024},
  doi           = {10.1016/j.jde.2024.07.017},
  issn          = {0022-0396}
}

@article{Floer1990,
  author        = {Floer, Andreas},
  title         = {A topological persistence theorem for normally hyperbolic manifolds via the {Conley} index},
  journal       = {Trans. Amer. Math. Soc.},
  volume        = {321},
  number        = {2},
  pages         = {647--657},
  year          = {1990},
  doi           = {10.2307/2001579}
}

@book{ComputationalHomologyTextbook,
  author        = {Kaczynski, Tomasz and Mischaikow, Konstantin and Mrozek, Marian},
  title         = {Computational Homology},
  series        = {Applied Mathematical Sciences},
  volume        = {157},
  year          = {2004},
  publisher     = {Springer-Verlag},
  address       = {New York},
  doi           = {10.1007/b97315},
  isbn          = {978-0-387-40853-8}
}

@unpublished{Robbin1989,
  author        = {Robbin, Joel W. and Salamon, Dietmar A. and Zeeman, Erik Christopher},
  title         = {Morse Inequalities and Zeta Functions},
  year          = {1989},
  url           = {https://people.math.ethz.ch/~salamon/PREPRINTS/zeta.pdf},
  note          = {Preprint, University of Wisconsin-Madison and University of Warwick}
}

@article{Salamon1985,
  author        = {Salamon, Dietmar A.},
  title         = {Connected simple systems and the {Conley} index of isolated invariant sets},
  journal       = {Trans. Amer. Math. Soc.},
  volume        = {291},
  number        = {1},
  pages         = {1--41},
  year          = {1985},
  doi           = {10.2307/1999893}
}

@book{Rybakowski1987,
  author        = {Rybakowski, Krzysztof P.},
  title         = {The Homotopy Index and Partial Differential Equations},
  series        = {Universitext},
  year          = {1987},
  publisher     = {Springer-Verlag},
  address       = {Berlin Heidelberg}
}

@article{Izhikevich2003,
  author        = {Izhikevich, Eugene M.},
  title         = {Simple model of spiking neurons},
  journal       = {IEEE Trans. Neural Netw.},
  volume        = {14},
  number        = {6},
  pages         = {1569--1572},
  year          = {2003},
  doi           = {10.1109/TNN.2003.820440}
}

\end{document}